\documentclass{article}
\usepackage[bookmarks=true]{hyperref}
\usepackage{amssymb}
\usepackage{amsmath}
\usepackage{amsthm}
\usepackage{framed}
\usepackage{marginnote}

\usepackage{color}
\usepackage{comment}
\usepackage{graphicx}
\usepackage{enumitem}
\usepackage{bookmark}
\usepackage{xcolor}
\usepackage[mathscr]{eucal}
\usepackage{multicol}
\usepackage[top=2cm,bottom=2cm]{geometry}
\usepackage{calc}
\usepackage[normalem]{ulem}
\usepackage{exmath}
\usepackage[capitalize,nameinlink]{cleveref}
\usepackage{pgfplots}
\usepackage{pgfplotstable}
\usepackage{placeins}  %

\usepackage{todonotes}
\usetikzlibrary{matrix}
\pgfplotsset{compat=newest}

\allowdisplaybreaks[1]

\graphicspath{{imgs/}}
\renewcommand{\vec}[1]{{\ensuremath{{\boldsymbol #1}}}}

\hypersetup{
    colorlinks,
    linkcolor={blue!50!black},
    citecolor={green!50!black},
    urlcolor={red!80!black}
}
\makeatletter
\@ifclassloaded{siamart190516}{
  \usepackage{silence}
  \WarningFilter{latex}{Marginpar on page}

  \newsiamthm{assumption}{Assumption}
  \newsiamthm{remark}{Remark}
  \crefname{equation}{}{}
  \crefname{theorem}{Theorem}{Theorems}
  \crefname{corollary}{Corollary}{Corollaries}
  \crefname{figure}{Figure}{Figures}
  \crefname{lemma}{Lemma}{Lemmas}
  \crefname{subsection}{Section}{Sections}
  \crefname{section}{Section}{Sections}
  \crefname{algocf}{Algorithm}{Algorithms}
  \crefname{assumption}{Assumption}{Assumptions}
  \crefname{proposition}{Proposition}{Propositions}
}{
  \newtheorem{theorem}{Theorem}[section]
  \newtheorem{remark}[theorem]{Remark}
  \newtheorem{lemma}[theorem]{Lemma}
  \newtheorem{proposition}[theorem]{Proposition}
  \newtheorem{assumption}[theorem]{Assumption}
  \newtheorem{corollary}[theorem]{Corollary}
  
  \newcommand{\email}[1]{\protect\href{mailto:#1}{#1}}

  \newenvironment{keywords}{\textbf{Keywords:}}{}
  \newenvironment{AMS}{\textbf{AMS Classification:}}{}

  \crefname{equation}{}{}
}

\usepackage{xpatch}
\xpatchbibdriver{article}
  {\usebibmacro{doi+url}}
  {\usebibmacro{doi+url}%
   \newunit\newblock
   \usebibmacro{eprint}}
  {}
  {}

\makeatletter%
\newcommand\defconstlist[2]{%
  \@ifundefined{rc@constlist@#1}{%
    \expandafter\newcommand\csname rc@useconst@#1\endcsname[1]{#2}%
    \expandafter\newcount\csname rc@constcounter@#1\endcsname%
    \expandafter\global\csname rc@constcounter@#1\endcsname=1\relax%
  }%
  {\errmessage{Constant list was already defined}}}%

\defconstlist{deflist}{\ensuremath{C_{#1}}}
\newcommand{\defaultconstlist}{deflist}

\newcommand\defconst[2][\defaultconstlist]{%
  \@ifundefined{rc@const@#1@#2}{%
      \expandafter\edef\csname rc@const@#1@#2\endcsname{%
      \noexpand\csname rc@useconst@#1\endcsname{\the\csname rc@constcounter@#1\endcsname}%
    }%
    \global\expandafter\advance\csname rc@constcounter@#1\endcsname1\relax%
  }%
  {\@warning{Constant was already defined}}%
}

\newcommand\const[2][\defaultconstlist]{%
  \@ifundefined{rc@const@#1@#2}{%
    \@ifundefined{rc@const@@custom@#2}{%
      \@warning{Undefined constant}??%
    }%
    {\csname rc@const@@custom@#2\endcsname}%
  }{\csname rc@const@#1@#2\endcsname}%
}

\newcommand\defcustomconst[2]{%
  \@ifundefined{rc@const@@custom@#1}{%
    \expandafter\def\csname rc@const@@custom@#1\endcsname{\noexpand#2}%
  }%
  {\@warning{Constant was already defined}}%
}
\makeatother

\def\kernel{\kappa}
\defcustomconst{C}{\ensuremath{C}}

\long\def\changed#1{{\color{red}{#1}}}
\def\groneq{Gr\"onwall's inequality}
\def\ito{It\^o}
\def\holdereq{H\"older's inequality}
\def\P{n}
\def\vecZspace{\rset^{d \times \P}}

\def\gdrift{\nu}
\def\gdiffusion{\varsigma}

\def\drift{a}
\def\diffusion{\sigma}
\def\law{\mu}

\def\X#1#2{X_{#1}^{#2}}

\def\Z#1#2{Z_{#1}^{#2}}

\def\eqand{\mathrm{and}\qquad}
\def\eqperiod{\mathrm{\ .}}
\def\eqcomma{\mathrm{\ ,}}

\def\Cbspace{\overline C_{b}}
\DeclarePairedDelimiter\seminorm{\lvert}{\rvert}

\makeatletter
\newcommand{\pushright}[1]{\ifmeasuring@#1\else\omit\hfill$\displaystyle#1$\fi\ignorespaces}
\newcommand{\pushleft}[1]{\ifmeasuring@#1\else\omit$\displaystyle#1$\hfill\fi\ignorespaces}

\newcounter{@countComma@counter}
\def\countComma#1{%
  \setcounter{@countComma@counter}{0}%
  \expandafter\@countComma#1,\@@countComma,\relax}
\def\@countComma#1,#2\relax{\ifx\@@countComma#1%
  \addtocounter{@countComma@counter}{-1}%
  \else\@countComma#2\relax%
  \stepcounter{@countComma@counter}\fi}
\def\replaceComma#1{%
  \expandafter\@replaceComma#1,\@@replaceComma,\relax}
\def\@replaceComma#1,#2\relax{\ifx\@@replaceComma#1%
  \else\partial#1\@replaceComma#2\relax\fi}

\ExplSyntaxOn
\NewDocumentCommand{\pderv}{smm}
{
  \IfBooleanTF{#1}
  {%
    \countComma{#3}%
    \frac{\partial%
      \ifnum\value{@countComma@counter}>0%
      {\stepcounter{@countComma@counter}
        ^\the@countComma@counter}%
      \fi%
      #2}{\replaceComma{#3}}%
  }
  {%
    \@@pderv:nn { #2 } { #3 }
  }
}

\NewDocumentCommand{\derv}{smm}
{
  \IfBooleanTF{#1}
  {%
    \frac{\D #2}{\D #3}
  }
  {%
    \@@pderv:nn { #2 } { #3 }
  }
}

\tl_new:N \@@pderv_function_tl
\tl_new:N \@@pderv_variable_tl

\cs_new_protected:Nn \@@pderv:nn
{
  \regex_match:nnTF { _ } { #1 }
  {%
    \tl_set:Nn \@@pderv_function_tl { #1 }
    \tl_set:Nn \@@pderv_variable_tl { #2 }
    \regex_replace_once:nnN
    { _ (.*) \Z } %
    { \c{sb}\cB\{\1,\u{@@pderv_variable_tl}\cE\} } %
    \@@pderv_function_tl
    \tl_use:N \@@pderv_function_tl
  }
  {%
    #1\sb{#2}
  }
}
\ExplSyntaxOff

\makeatother

\author{Abdul-Lateef Haji-Ali\thanks{Maxwell Institute, Heriot-Watt University, Edinburgh, UK.
    (\email{a.hajiali@hw.ac.uk}).}%
  \and H\r{a}kon Hoel\thanks{%
    University of Oslo, Norway. (\email{haakonah@math.uio.no}).}%
  \and Ra\'ul Tempone \thanks{ RWTH Aachen University, Germany and King
    Abdullah University of Science \& Technology (KAUST), Thuwal, Saudi Arabia
    ({\email{raul.tempone@kaust.edu.sa}}).}%
}

\def\todo#1{}
\def\changed#1{#1}
\hidedetails
\title{Weak convergence analysis in the particle limit of the McKean--Vlasov
  equations using stochastic flows of particle systems}

\begin{document}

\maketitle

\begin{abstract}
  We present a proof showing that the weak error of a system of \(\P\)
  interacting stochastic particles approximating the solution of the
  McKean--Vlasov equation is \(\Order{\P^{-1}}\). Our proof is based on the
  Kolmogorov backward equation for the particle system and bounds on the
  derivatives of its solution which we derive more generally using the
  variations of the stochastic particle system. The convergence rate is
  verified by numerical experiments which also indicate that the assumptions
  made here and in the literature can be relaxed.
\end{abstract}



\section{Introduction}\label{sec:intro}%
For \(\drift \: : \: \rset^{d} \times \rset \to \rset^{d}\),
\(\kernel_{1},\kernel_{2} \: : \: \rset^{d} \times \rset^{d} \to \rset\) and
\(\diffusion \: : \: \rset^{d} \times \rset \to \rset^{d \times d'}\), \(t \geq 0\) and
\(\br{W\p{s}}_{s \geq 0}\) being a \(d'\)-dimensional Wiener process defined on
some probability space with the natural filtration, we consider the following
McKean--Vlasov equation:
\begin{equation}\label{eq:Z-MV}
  \begin{aligned}
    Z\p{t} = \xi &&+& \int_{0}^{t} \drift \p*{Z\p{s}, \int_{\rset^{d}} \kernel_{1}\p{Z\p{s}, \vec z} \law_{s}\p{\!\D \vec z} } \D s \\
               &&+& \int_{0}^{t} \diffusion\p*{Z\p{s}, \int_{\rset^{d}} \kernel_{2}\p{Z\p{s}, \vec z} \law_{s}\p{\!\D \vec z}} \D W\p{s}\eqcomma
  \end{aligned}
\end{equation}
where \(\law_{s}\) denotes the law of \(Z\p{s}\) for all \(s \geq 0\), and \(\xi\)
denotes a random initial state whose law is \(\law_{0}\) and is assumed to be
independent of the Wiener process, \(W\). We focus on one-dimensional
interaction kernels \(\kernel_{1},\kernel_{2} \: : \: \rset^{d} \times \rset^{d} \to
\rset\) for clarity of presentation since high-dimensional kernels can be
treated in a similar way; see \cref{rem:multi-kernels}.
For \(\drift\p{x, y} = y\), with \(\diffusion\) being constant and
\(\kernel_{1}\) being bounded and Lipschitz, \cite[Theorem
I.1.1]{sznitman1991topics} shows the existence and uniqueness of a strong
solution to \cref{eq:Z-MV}. A recent analysis yielded the same result in
\cite{mishura:mkvlasov-existence} when the initial condition \(\xi\) has a
finite fourth moment, \(\drift\p{x,y}=\diffusion\p{x,y}=y\), under
non-degeneracy conditions on \(\kernel_{2}\), and when for all \(\vec x,\vec
x',\vec y \in \rset^{d}\), there exists a constant \(C>0\) such that
\[
  \begin{aligned}
    \abs{\kernel_2\p*{\vec x,\vec y}}+\abs{\kernel_2\p*{\vec x,\vec y}} &\leq C\p{1+\norm{\vec x}} \eqcomma\\
    \abs{\kernel_2\p*{\vec x,\vec y}-\kernel_2\p*{\vec x',\vec y}} &\leq C\,\p{1+\norm{\vec y}^{2}}\,\norm{\vec x-\vec x'} \eqperiod
  \end{aligned}
\] {The work \cite{crisan:approximate-spdes} also shows the existence and
  uniqueness result for \(\kernel_{2}=0\) and a particular form of \(\drift\).
  Existence of weak solutions was also shown in \cite{lukasz:lyapunov} under
  certain measure-dependent Lyapunov conditions. %
}
In the current work, we do not focus on existence of solutions to
\cref{eq:Z-MV} and instead assume the existence of weak solutions and consider
approximations of \(Z\) using a system of \(n\) \ito{} stochastic differential
equations~(SDEs), also known as an interacting stochastic particle system,
with pairwise interaction kernels:
\begin{equation}\label{eq:X-sys}
  \begin{aligned}
    X_i^{\P}\p{t} = \xi_{i}
    &&+& \int_{0}^{t} \drift\p*{X_i^{\P}\p{s}, \frac{1}{\P} \sum_{j=1}^{\P} \kernel_{1}\p*{X_i^{\P}\p{s},X_j^{\P}\p{s}}} \D s \\
    &&+& \int_{0}^{t} \diffusion\p*{X_{i}^{\P}\p{s}, \frac{1}{\P} \sum_{j=1}^{\P} \kernel_{2}\p{X_i^{\P}\p{s}, X_j^{\P}\p{s}}} \D W_i\p{s} \eqcomma
  \end{aligned}
\end{equation}
for $i \in \br{1, 2, \ldots, \P}$, where \(\xi_{i}\) are i.i.d. and have the same law,
\(\law_{0}\), and \(\br{W_{i}\p{s}}_{s \geq 0}\) are independent
\(d'\)-dimensional Wiener processes and independent of
\(\br{\xi_{i}}_{i=1}^{\P}\). In other words, the law \(\law_{t}\) for \(t\geq 0\)
is approximated by an empirical measure based on the particles \(\vec
X^{\P}\p{t} \defeq \br{X_{i}^{\P}\p{t}}_{i=1}^{\P}\). It should be noted that
these particles are identically distributed but not independent.

For \(\vec Z^{n} \defeq \p{Z_{i}}_{i=1}^{\P}\) being \(\P\) independent
samples of the solution to the McKean--Vlasov equation \eqref{eq:Z-MV} and a
function \(g : \vecZspace \to \rset\), the weak error at time \(t\) is
defined as the absolute difference \(\abs{\E{g\p{\vec X^{\P}\p{t}}} -
  \E{g\p{\vec Z^{n}\p{t}}}}\). The weak error was established to be
\(\Order{1/\P}\) in, e.g., ~\cite[Chapter 9]{kolokoltsov:book-nonlinear}
and~\cite[Theorem 6.1]{mischler2015new}%
.
These works assume that \(\kernel_{2} \equiv 0\) and build upon semigroup theory in
measure-valued function spaces to prove their results. {On the other hand, the
  work \cite{jordain:bias-particle} employs a similar methodology to the
  current work but assumes that \(\kernel_{1}\) and \(\kernel_{2}\) in
  \eqref{eq:Z-MV} do not depend on the state \(Z\)}.
There is an increasing interest in extending proofs of strong and weak
convergence in more general settings with nonlinear drift/diffusion
coefficients. To that end, works such as ~\cite[Theorem
B.2]{szpruch:antithetic}, ~\cite[Theorem 2.17]{chassagneux:weak} and
\cite{deraynal:backward-kolmogorov} use
Lions-derivatives~\cite{cardaliaguet:mean-field-notes} to bound derivatives
with respect to measures and a master equation for probability
measures~\cite{cardaliaguet:master,chassagneux2014probabilistic}. For a more
exhaustive literature review, see \cite{chassagneux:weak}.

In \cref{sec:weak-conv}, we present a new method to show the rate of weak
convergence. The principal steps %
involve using the Kolmogorov backward equation to represent the weak error and the
stochastic flows and the dual functions to bound the weights in the resulting
dual weighted residual representation.
Using the Kolmogorov backward equation to estimate the weak error in SDEs goes
back to the ideas of Talay and Tubaro \cite{TaTub90}, who estimated the time
discretization error for uniform deterministic time-steps. The works
\cite{BALLY199535,BallyTa96}, extended the analysis to approximations with
non-smooth observables and the probability density of the solution at a given
time. Kloeden and Platen \cite{kloden:numsde} generalized the analysis in
\cite{TaTub90} to weak approximations of a higher order. Later, in a series of
works inspired by \cite{TaTub90}, the authors developed methods based on
stochastic flows and dual functions to bound the weights in the resulting
dual-weighted residual representation. This approach provided the analysis for
the weak approximation of SDEs using non-uniform, possibly stochastic,
time-steps, see
\cite{szepessy2001adaptive,moon2005convergence,mordecki2008adaptive,bayer2010adaptive}.
Furthermore, the same analysis line was also used for adaptive Multilevel
Monte Carlo, \cite{hoel2014implementation,hoel2016construction}.

The closest inspiration for deriving the weak convergence rate goes back to
the use of the mentioned techniques in the context of multiscale
approximation. Those works derived macroscopic SDEs continuum models by
choosing their drift and diffusion functions to minimize the weak error in the
given macroscopic observables when compared with a given base model.
Particularly, \cite{katsoulakis:limparticle} employed master equations with
long-range interaction potentials as a base model as the stochastic Ising
model with Glauber dynamics, whereas \cite{Schwerin10} determined the
stochastic phase-field models from atomistic formulations by coarse-graining
molecular dynamics to model the dendritic growth of a crystal in an
under-cooled melt. Systems of coupled SDEs with increasing size could also be
useful for approximating a non-Markovian behavior. For example, the recent
work \cite{bayer2020weak} investigated weak convergence rates for a rough
stochastic volatility model emerging in mathematical finance, namely, the
rough Bergomi model. As in this study, the analysis in \cite{bayer2020weak}
also employed a dual-weighted representation of the weak error, yielding an
error expansion that characterizes the weak convergence rate.
{A similar method was also used in \cite{jordain:bias-particle} for a special
  case of \eqref{eq:Z-MV} in which \(\kernel_{1}\) and \(\kernel_{2}\) do not
  depend on the state \(Z\), i.e., \(\kernel_{\cdot}\p{x,y}=\kernel_{\cdot}\p{y}\). A
  key difference to our methodology is that in \cite{jordain:bias-particle},
  the Kolmogorov backward equation involves \(\law_{t}\), the law of
  \(Z\p{t}\) while we instead utilize the Kolmogorov backward equation for the
  particle system \eqref{eq:X-sys}}.

In \cref{sec:sde-variation-bounds} we prove the technical results that are
needed for the preceding analysis. In particular, we determine sufficient
conditions to bound derivatives of the solution to the Kolmogorov backward
equation for a generic multidimensional SDE by bounding moments of the first,
second, and third variations of the
SDE. %
Finally, in \cref{sec:numerics} we numerically study the weak error of a
particle approximation to the solution of the McKean-Vlasov equation. In
particular, we show numerically that the weak convergence rate is the same for
an example stochastic particle system that does not satisfy the regularity
conditions of our theory or those of others in the literature cited above.
Therefore, further work is necessary to extend the current approaches.

In what follows, we will use the notation \(A \lesssim B\) to denote that there is a
constant \(0 < c < \infty\) which is independent of \(\P\), the size of the
particle system \cref{eq:X-sys}, such that \(A \leq c B\).
For a multi-index \(\vec \ell \in \nset^{\P}\), \(n \in \nset\), define the
derivative
\[
  \frac{\partial^{\abs{\vec \ell}}}{\partial x^{\vec \ell}} \defeq
  \frac{\partial^{\abs{\vec \ell}}}{\prod_{j=1}^{\P} \partial x_{j}^{\ell_{j}}}\eqcomma
\]
where \(\abs{\vec \ell} = \sum_{i=1}^{n} \ell_{i}\). Let \(\norm{\cdot}\) denote the Euclidean norm and let
\(C\p{\rset^{n} ; \rset^{m}}\) denote the space of continuous functions \(u \equiv
\p{u_{i}}_{i=1}^{m} : \rset^{n} \to \rset^{m}\) for which the \changed{extended}
norm
\[
  \norm{u}_{C\p{\rset^{n} ; \rset^{m}}} \defeq \sum_{j=1}^{m} \sup_{\vec x \in
    \rset^{m}} \abs{u_{j}\p*{\vec x}}
\]
is finite. Let also \(C\p{\rset^{n} ; \rset} = C\p{\rset^{n}}\). When \(u\)
has continuous derivatives up to order \(k\), define the %
\changed{extended} semi-norm
\begin{equation*}
  \seminorm{u}_{\Cbspace^{k}\p{\rset^{n} ; \rset^{m}}} =
  \sum_{j=1}^{m} \: \sum_{\vec \ell \in \nset^{\P}, 1 \leq \abs{\vec \ell} \leq k}
  \norm*{\frac{\partial^{\abs{\vec \ell}} u_{j}}{\partial \vec x^{\vec \ell} }}_{C\p{\rset^{n} \,; \,\rset^{m}}}\eqperiod
\end{equation*}
and the norm \(\norm{u}_{C^k\p{\rset^{n} ; \rset^{m}}} =
\norm{u}_{C\p{\rset^{n} ; \rset^{m}}} + \seminorm{u}_{\Cbspace^{k}\p{\rset^{n}
    ; \rset^{m}}} \). For a vector \(\vec x \in \vecZspace\), we denote its
components as \(\vec x = \p*{x_{i,j}}_{i=\br{1, \ldots, \P}, j=\br{1, \ldots, d}} \in
\vecZspace\). Similarly for a function \(u \: : \: \vecZspace \to \rset\), we
will use the notation \(\nabla u = \p*{\frac{\partial u}{\partial x_{i,j}}}_{i=\br{1, \ldots, \P},
  j=\br{1, \ldots, d}}\) for its gradient.
 
\section{A Bound on the Weak Error}\label{sec:weak-conv}%
\def\Ddrift{\overline \drift}
In this section, we prove that the weak error as defined in the introduction
is \(\Order{1/n}\). We start by stating boundedness and convergence results
involving only samples of \(Z\), the solution to the McKean-Vlasov equation
\cref{eq:Z-MV}.
\begin{proposition}
  \changed{Assume that weak solutions to \eqref{eq:Z-MV} exist} and let
  \(\br{Z_{i}}_{i=1}^{\P}\) be \(\P\) independent processes each satisfying
  \cref{eq:Z-MV} with independent underlying Wiener processes. Let \(\kernel :
  \rset^{d} \times \rset^{d} \to \rset\) be a Lipschitz continuous function, i.e.,
  there exists a constant \(C\) such that
  \[
    \abs{\kappa\p{\vec x, \vec y} - \kappa\p{\vec x', \vec y'}} \leq C \p*{\norm{\vec x - \vec x'} + \norm{\vec y - \vec y'}}
    \qquad \textnormal{for all } \vec x, \vec x', \vec y, \vec y' \in \rset^{d}\eqperiod
  \]
  Let \(p \in \br{1, 2, \ldots} \), then for any \(i \in \br{1, \ldots, \P} \), we have
  \begin{equation}\label{eq:p-mom-kernel-sum-strong}
    {\E*{\abs*{\frac{1}{\P}\sum_{j=1}^{\P}
          \kernel\p{Z_{i}\p{t}, Z_{j}\p{t}} - \int_{\rset^{d}} \kernel\p{Z_{i}\p{t}, \vec z}
          \law_{t}\p{\!\D \vec z}
        }^{2p}}}
    \lesssim \detailed{2^{2p}\, p\, C^{2p} \, c_{2p} \, }
    \P^{-p}
    \, \E*{\norm[\big]{\Z{}{}\p{t}}^{2p}}\eqperiod
  \end{equation}
  Moreover, assuming that \(a, \sigma, \kernel_{1}\) and \(\kernel_{2}\) in
  \cref{eq:Z-MV} are Lipschitz continuous, we have
  \begin{equation}
    \label{eq:Z-moments}
    \sup_{0\leq s \leq t}\E{\norm{\Z{}{}\p{s}}^{2p}} \lesssim 1+\E{\norm{\xi}^{2p}} \eqperiod
  \end{equation}
  The hidden constants in \cref{eq:Z-moments,eq:p-mom-kernel-sum-strong}
  depend only on \(d, p\), \(t\), and the Lipschitz constants.
\end{proposition}
\begin{proof}
  The moment boundedness result \cref{eq:Z-moments} with Lipschitz assumptions
  is classical \cite{sznitman1991topics} using It\^o's formula, Young's and
  Gr\"onwall's inequalities.
  For \(p=1\), we set, without loss of generality, \(i=1\) and let
  \begin{equation}\label{eq:delta-kernel}
    \Delta_{j} \kernel
    \defeq \int_{\rset^{d}} \kernel\p{Z_{1}\p{t}, \vec z}
    \law_{t}\p{\!\D \vec z} - \kernel\p{Z_{1}\p{t}, Z_{j}\p{t}}\eqperiod
  \end{equation}
  Expanding the
  square
  \[
    \begin{aligned}
      \E*{\p*{\frac{1}{\P}\sum_{j=1}^{\P} \Delta_{j} \kernel}^{2}}
      &=
        \frac{1}{\P^{2}} \sum_{j=1}^{\P} \sum_{j'=1}^{\P} {\E*{ \Delta_{j} \kernel \Delta_{j'} \kernel }}\eqperiod
    \end{aligned}
  \]
  When \(j\neq j'\) and are both different from 1,
  we have
  \[
    \E*{\Delta_{j} \kernel \Delta_{j'} \kernel} = \E*{\E*{\Delta_{j} \kernel\given
        \Z{1}{}\p{t}} \, \E*{\Delta_{j'} \kernel\given \Z{1}{}\p{t}}}
    = 0\eqcomma
  \]
  since, for a given \(\Z{1}{}\p{t}\), \(\Delta_{j} \kernel\) and \(\Delta_{j'}
  \kernel\) are conditionally independent and \({\E{\Delta_{j} \kernel \given
      \Z{1}{} \p{t}} = 0}\) when \(j \neq 1\). When \(j=j'\) or when \(1 \in
  \br{j,j'}\), we bound using \holdereq{} and the Lipschitz assumption on
  \(\kernel\),
  \[
    \begin{aligned}
      \abs{\E*{\Delta_{j} \kernel \Delta_{j'} \kernel}}
      &\leq \E*{\p{\Delta_{j} \kernel}^{2}}\\
      &\leq \E*{\p*{\int_{\rset^{d}} C \, \norm*{\vec z - Z_{j}\p{t}} \,  \mu_{t}\p{\D \vec z}}^{2}} \\
      \begin{details}
        &\leq C^{2} \p*{\E{ \p{\E{\norm{Z}} + \norm{Z}}^{2}}} \\
        &\leq 2 C^{2}  \p{\E{\norm{Z}}^{2} + \E{\norm{Z}^{2}}} \\
      \end{details}
      &\leq 4 C^{2}\: \E*{\norm[\big]{\Z{}{}\p{t}}^{2}} \eqperiod
    \end{aligned}
  \]
  Substituting back yields the claim.
  \begin{details}
    We have
    \begin{equation}\label{eq:exist-kernel-sum-strong}
      \abs*{\E*{\p*{\frac{1}{\P}\sum_{j=1}^{\P} \Delta_{j} \kernel}^{2}}}
      \le
      \frac{12 C^{2}}{\P}
      \, \E*{\norm[\big]{\Z{}{}\p{s}}^{2}}\eqperiod
    \end{equation}
  \end{details}
  The proof for \(p>1\) is fundamentally similar using the multinomial
  theorem.%
\end{proof}

\begin{lemma}\label{lem:weak-Z-rate}
  \changed{Assume that weak solutions to \eqref{eq:Z-MV} exist} and let \(\vec
  Z^{\P} = \br{Z_{i}}_{i=1}^{\P}\) be \(\P\) independent processes each
  satisfying \cref{eq:Z-MV} with independent underlying Wiener processes,
  \(\kernel : \rset^{d} \times \rset^{d} \to \rset\) be a Lipschitz continuous
  function and \(f : \rset^{d} \times \rset \to \rset\) be such that
  \begin{equation}\label{eq:b-linear-bound}
    \abs*{\frac{\partial f\p*{\vec x, y}} {\partial y}}
    + \abs*{\frac{\partial^{2} f\p*{\vec x, y}} {\partial y^{2}}}
    \leq  \widetilde C \p{1 + \norm{\vec x} + \abs{y}}\eqcomma
  \end{equation}
  for all \(\vec x \in \rset^{d}\) and \(y \in \rset\). Then, for any \(g \in
  C^{1}(\vecZspace)\) and any \(i \in \br{1, \ldots, n}\) we have
  \begin{equation}\label{eq:b_i-bound-result}
    \begin{aligned}
      &\abs*{\E*{\p*{
      f\p*{Z_{i}\p{t}, \int_{\rset^{d}} \kernel\p{Z_{i}\p{t}, \vec z} \law_{t}\p{\!\D \vec z}}
      - f\p*{Z_{i}\p{t}, \frac{1}{\P}\sum_{j=1}^{\P} \kernel\p{Z_{i}\p{t}, Z_{j}\p{t}}}
        }
        g\p*{\vec Z^{\P}\p{t}}}} \\
      &\lesssim \P^{-1}  \: \norm*{g}_{C^{1}\p{\vecZspace}} \, \p{1 + \E{\norm{Z\p{t}}^{4}}}
        \eqperiod
    \end{aligned}
  \end{equation}
\end{lemma}
\begin{proof}
  Without loss of generality, we fix \(i=1\) and define
  \begin{equation*}
    \Delta f \defeq f\p*{Z_{1}\p{t}, \int_{\rset^{d}} \kernel\p{Z_{1}\p{t}, \vec z} \law_{t}\p{\!\D \vec z}}
          - f\p*{Z_{1}\p{t}, \frac{1}{\P}\sum_{j=1}^{\P} \kernel\p{Z_{1}\p{t}, Z_{j}\p{t}}} \eqperiod
  \end{equation*}
  By Taylor expanding, we can bound
  \def\Ddiff{\overline{f}}
  \def\DDdiff{\overline{\overline{f}}}
  \begin{equation}\label{eq:b_i-bound}
    \begin{aligned}
      \abs*{\E*{\Delta f \: g\p{\vec Z^{\P}\p{t}}}}
      \leq&&&
           \abs*{\E*{\Ddiff \p{Z_{1}\p{t}} \: \p*{
           \frac{1}{\P} \sum_{j=1}^{\P} \Delta_{j} \kernel} \,
           g\p{\vec Z^{\P}\p{t}}}} \\
       &&+&
            \, \norm*{g}_{C\p{\vecZspace}}
            \E*{
            \abs*{\DDdiff\p{Z_{1}\p*{t}}}
            \p*{ \frac{1}{\P} \sum_{j=1}^{\P}\Delta_{j} \kernel}^{2}}\eqcomma
    \end{aligned}
  \end{equation}
  where \(\Delta_{j} \kernel\) is as defined in \eqref{eq:delta-kernel} and, for all \(\vec x \in \rset^{d}\),
  \[
    \begin{aligned}
      \Ddiff \p{\vec x} &\defeq \frac{\partial f}{\partial y} \p*{\vec x, \int_{\rset^{d}} \kernel\p{\vec x, \vec z} \law_{t}\p{\!\D \vec z}}\eqcomma\\
      \eqand \DDdiff\p{\vec x} &\defeq \int_{0}^{1}
                                 \frac{\partial^{2} f}{\partial y^{2}}
                                 \p*{\vec x,
                                 s \p*{\frac{1}{\P} \sum_{j=1}^{\P} \Delta_{j} \kernel}
                                 + \int_{\rset^{d}} \kernel\p{\vec x, \vec z} \law_{t}\p{\!\D \vec z}
                                 } \p*{1-s}\; \D s\eqperiod
    \end{aligned}
  \]
  By \cref{eq:b-linear-bound} and \(\kernel\) being Lipschitz continuous,
  implying linear growth, we
  have %
  for \(\vec x \in \rset^{d}\),
  \[
    \begin{aligned}
      \abs{\Ddiff \p{\vec x}
}
      \begin{details}
        &\leq \widetilde C \,\p*{1+ \abs{\vec x} + C \p{1+\abs{\vec x}
            +\E{\norm{Z\p*{t}}}}}\\
      \end{details}
      &\lesssim 1 + \norm{\vec x}
 + \E{\norm{Z\p*{t}}}\\
      \abs{\DDdiff\p{\vec x}
}
      \begin{details}
        &\leq \int_{0}^{1} \abs*{ \frac{\partial^{2} f^{2}}{\partial y^{2}} \p*{\vec x, s
            \p*{\frac{1}{\P} \sum_{j=1}^{\P} \Delta_{j} \kernel}
            +\int_{\rset^{d}} \kernel\p{\vec x, \vec z} \law_{t}\p{\!\D \vec z}}} \p*{1-s}  \D s\\
        &\leq \widetilde C \int_{0}^{1} \p*{ 1 + \norm{\vec x} + s
          \abs*{\frac{1}{\P} \sum_{j=1}^{\P} \Delta_{j} \kernel} +\abs*{\int_{\rset^{d}}
            \kernel\p{\vec x, \vec z} \law_{t}\p{\!\D \vec z}}
        } \p*{1-s} \D s\\
        &\leq \frac{\widetilde C}{2} \p*{ 1 + \norm{\vec x} + \frac{1}{3}
          \abs*{\frac{1}{\P} \sum_{j=1}^{\P} \Delta_{j} \kernel} + \;
          \abs*{\int_{\rset^{d}} \kernel\p{\vec x, \vec z} \law_{t}\p{\!\D \vec z}}
        }\\
        &\leq \widetilde C \p*{ 1 + \norm{\vec x} + \frac{1}{6}
          \abs*{\frac{1}{\P} \sum_{j=1}^{\P} \Delta_{j} \kernel} + C\; \frac{1}{2} \;
          \p*{1 + \norm{\vec x} + \E{\norm{Z\p{t}}}}
        }  \\
      \end{details}
      &\lesssim
        1
        + \abs*{\frac{1}{\P} \sum_{j=1}^{\P} \Delta_{j} \kernel}
        + \norm{\vec x}
        + \E*{\norm*{Z\p{t}}} \eqperiod
    \end{aligned}
  \]
  By \holdereq{} and \cref{eq:p-mom-kernel-sum-strong}, we have
  \begin{equation*}
    \begin{aligned}
      \E*{\abs*{\DDdiff\p{Z_{1}\p{t}}} \p*{ \frac{1}{\P} \sum_{j=1}^{\P}\Delta_{j} \kernel}^{2}}
      \lesssim &&& \p{1 + \E{\norm{Z\p*{t}}}} \: \E*{\p*{ \frac{1}{\P} \sum_{j=1}^{\P}\Delta_{j} \kernel}^{2}}+
            \p*{\E*{\p*{ \frac{1}{\P} \sum_{j=1}^{\P}\Delta_{j} \kernel}^{4}}}^{3/4}\\
        &&+& \p*{\E*{\norm{Z\p*{t}}^{2}}}^{1/2} \p*{\E*{\p*{ \frac{1}{\P} \sum_{j=1}^{\P}\Delta_{j} \kernel}^{4}}}^{1/2}\\
      \lesssim&&& \p*{1 + \E*{\norm{Z\p*{t}}^{4}}} \: \P^{-1} \eqperiod
    \end{aligned}
  \end{equation*}
  Furthermore,
  \begin{equation}\label{eq:b_i-bound-term-1}
    \begin{aligned}
      \abs*{\E*{\Ddiff \p{Z_{1}\p{t}} \: \p*{ \frac{1}{\P} \sum_{j=1}^{\P} \Delta_{j} \kernel} \, g\p*{\vec Z^{\P}\p{t}}}}
      \leq&&& \frac{1}{\P} \; \abs*{\E*{\Ddiff \p{Z_{1}\p{t}} \: \p*{
           \Delta_{1} \kernel} \, g\p{\vec Z^{\P}\p{t}}}}\\
       &&+& \frac{1}{\P} \sum_{j=2}^{\P} \abs*{\E*{\Ddiff \p{Z_{1}\p{t}} \: \Delta_{j} \kernel \,
            g\p{\vec Z^{\P}\p{t}} }}\eqperiod
    \end{aligned}
  \end{equation}
  Here by again using \cref{eq:b-linear-bound} and \(\kernel\) being Lipschitz continuous we
  have,
  \begin{equation}\label{eq:b_i-bound-term-1-1}
    \begin{aligned}
      \abs*{\E*{\Ddiff \p{Z_{1}\p{t}} \: \p*{ \Delta_{1} \kernel} \, g\p*{\vec Z^{\P}\p{t}}}}
      &\leq %
        \norm{g}_{C\p{\vecZspace}} \, \,
        \E*{\abs{\Ddiff \p{Z_{1}\p{t}}} \: \abs*{ \Delta_{1} \kernel}}\\
      \begin{details}
        &\leq \norm{g}_{C\p{\vecZspace}} \,
        \p{C \, \widetilde C }\, \E{\p*{1+ \abs{Z_{1}\p{t}} + C
            \p{1+\abs{Z_{1}\p{t}}+\E{\abs{Z_{1}\p{t}}}}} \, \p*{\abs{Z_{1}\p{t}} +
            \E{\abs{Z_{1}\p{t}}}}
        } \\
        &\lesssim \norm{g}_{C\p{\vecZspace}} \,
        \E{\p{1+\abs{Z_{1}\p{t}}+\E{\abs{Z_{1}\p{t}}}}^{2}} \\
        &\lesssim \norm{g}_{C\p{\vecZspace}}
        \p*{1 + \E{\abs{Z\p{t}}^{2}} + \p{\E{\abs{Z\p{t}}}}^{2}} \\
      \end{details}
      &\lesssim \norm{g}_{C\p{\vecZspace}}
        \p*{1 + \E{\abs{Z\p{t}}^{4}}}\eqcomma
    \end{aligned}
  \end{equation}
  and for \(j = \br{2, \ldots, \P}\), using a Taylor expansion of \(g\),
  \begin{equation}\label{eq:b_i-bound-term-1-2}
    \begin{aligned}
      &\abs*{\E*{\Ddiff \p{Z_{1}\p{t}} \: \Delta_{j} \kernel \: g\p*{\vec Z^{\P}\p{t}}}}
      \leq \abs*{\E*{ \Ddiff \p{Z_{1}\p{t}} \: \Delta_{j} \kernel\,
           g\p{\vec Z_{-j}^{\P}\p{t}}}} \\
      &\hskip 2cm + \abs*{\E*{\Ddiff \p{Z_{1}\p{t}} \: \Delta_{j} \kernel \:
          \p{\vec Z^{\P}\p*{t} - \vec Z^{\P}_{-j}\p*{t}}^{T}
          \int_{0}^{1} %
          \nabla g\p[\big]{s \vec Z^{\P}\p*{t} - \p{1-s} \vec Z^{\P}_{-j}\p*{t}}
          \D s}}\eqperiod
    \end{aligned}
  \end{equation}
  Here, \(\vec Z^{\P}_{-j}\p{t} = \p{Z_{1}\p{t}, \ldots, Z_{j-1}\p{t},0,
    Z_{j+1}\p{t}, \ldots, Z_{\P}\p{t}} \in \vecZspace\) is the same as \(\vec
  Z^{\P}\p{t}\) but with the \(\nth{j}\) entry replaced by 0. Note that
  \begin{equation}\label{eq:b_i-bound-term-1-2-1}
    \begin{aligned}
      \E*{ \Ddiff \p{Z_{1}\p{t}} \: \Delta_{j} \kernel\: g\p{\vec Z^{\P}_{-j}\p{t}}}
      &= \E*{ \Ddiff \p{Z_{1}\p{t}} \: \E[\big]{\Delta_{j} \kernel \given Z_{1}\p{t}}\, g\p{\vec Z^{\P}_{-j}\p{t}}}
      &= 0\eqcomma
    \end{aligned}
  \end{equation}
  as \(Z_{j}\p{t}\) has law \(\law_{t}\) and is independent of \(\vec
  Z^{\P}_{-j}\) and of \(Z_{1}\p{t}\). Using that \(\kernel\) is Lipschitz
  continuous, we bound
  \begin{equation}\label{eq:b_i-bound-term-1-2-2}
    \begin{aligned}
      &\abs*{\E*{\Ddiff \p{Z_{1}\p{t}} \: \Delta_{j} \kernel \:
        \p{\vec Z^{\P}\p*{t} - \vec Z^{\P}_{-j}\p*{t}}^{T}
        \int_{0}^{1} %
        \nabla g\p[\big]{s \vec Z^{\P}\p*{t} - \p{1-s} \vec Z^{\P}_{-j}\p*{t}}
        \D s}} \\
      &\leq
        \sum_{i=1}^{d} \norm*{\frac{\partial g}{\partial x_{j,i}}}_{C\p{\vecZspace}}
        \ \E*{
        \abs*{\Ddiff \p{Z_{1}\p{t}}} \:
        \abs{
        \Delta_{j} \kernel
        } \: \norm{Z_{j}\p{t}}}\\
      \begin{details}
        & \leq \sum_{i=1}^{d}\norm*{\frac{\partial g}{\partial x_{j,i}}}_{C\p{\vecZspace}} \ \, \p{C
          \widetilde C}\, \E{\p*{1+ \norm{Z_{1}\p{t}} + C
            \p{1+\norm{Z_{1}\p{t}}+\E{\norm{Z_{1}\p{t}}}}} \, \p*{\norm{Z_{j}\p{t}} +
            \E{\norm{Z_{1}\p{t}}}} \:
          \norm{Z_{j}\p{t}}}\\
        & \lesssim \sum_{i=1}^{d}\norm*{\frac{\partial g}{\partial x_{j,i}}}_{C\p{\vecZspace}} \ \E{
          \p{1+\norm{Z_{1}\p{t}}+\E{\norm{Z_{1}\p{t}}}} \, \p*{\norm{Z_{1}\p{t}}
            + \E{\norm{Z_{1}\p{t}}}} \:
          \norm{Z_{j}\p{t}}}\\
        &\lesssim \sum_{i=1}^{d} \norm*{\frac{\partial g}{\partial x_{j,i}}}_{C\p{\vecZspace}}
        \, \p*{1 + \E{\norm{Z_{1}\p{t}}^{3}} + \p*{\E{\norm{Z_{1}\p{t}}}}^{3} }\\
      \end{details}
      &\lesssim  \sum_{i=1}^{d}\norm*{\frac{\partial g}{\partial x_{j,i}}}_{C\p{\vecZspace}}
        \, \p*{1 + \E{\norm{Z\p{t}}^{4}} }
        \eqperiod
    \end{aligned}
  \end{equation}
  Substituting~\cref{eq:b_i-bound-term-1-2-1}
  and~\cref{eq:b_i-bound-term-1-2-2} into~\cref{eq:b_i-bound-term-1-2}, and
  the result and~\cref{eq:b_i-bound-term-1-1} into~\cref{eq:b_i-bound-term-1},
  and then substituting the result %
  into~\cref{eq:b_i-bound},
  we arrive at the claimed result.
  \begin{details}
    \begin{equation}
      \begin{aligned}
        \abs*{\E*{\Delta f \,
        \: g\p*{\vec Z^{\P}\p{t}}}
        } \lesssim \P^{-1} \;
        \p*{
        \norm*{g}_{C\p{\rset^\P }} +
        \sum_{j=2}^{\P} \sum_{i=1}^{d} \norm*{\frac{\partial g}{ \partial x_{j,i}}}_{C\p{\rset^\P }}}\eqperiod
      \end{aligned}
    \end{equation}
  \end{details}
\end{proof}
 
We now state the main result of the paper, which will also depend on the
technical bounds that will be derived in \cref{sec:sde-variation-bounds}.
\begin{theorem}[Weak convergence result]\label{thm:weak-conv}
  \changed{Assume that weak solutions to \eqref{eq:Z-MV} exist} and let \(\vec
  Z^{\P} = \br{Z_{i}}_{i=1}^{\P}\) be \(\P\) independent processes each
  satisfying \cref{eq:Z-MV} with independent underlying Wiener processes.
  Assume that
  \begin{equation}\label{eq:C3-boundedness}
    \begin{aligned}
      \seminorm{\drift}_{\Cbspace^{3}\p{\rset^{d} \times \rset \,  ; \, \rset^{d} }}
      + \seminorm{\diffusion}_{\Cbspace^{3}\p{\rset^{d}\times \rset \, ; \,\rset^{d \times d'}}}
      + \seminorm{\kernel_{1}}_{\Cbspace^{3}\p{\rset^{d}\times \rset^{d}}}
      + \seminorm{\kernel_{2}}_{\Cbspace^{3}\p{\rset^{d}\times \rset^{d}}} < \infty
         \eqcomma%
    \end{aligned}
  \end{equation}
  \changed{and let \(\vec X^{\P} = \br{\X{i}{\P}}_{i=1}^{\P}\) satisfy
    \cref{eq:X-sys}}. Then for \(g : \vecZspace \to \rset\) with continuous
  bounded derivatives up to the third order and any \(T>0\), we have
  \begin{equation}\label{eq:weak-convergence}
    \begin{aligned}
      \abs*{\E*{g\p{\vec X^{\P}\p{T}} - g\p{\vec Z^{\P}\p{T}}}}
      &\lesssim \p{1 + \E{\norm{\xi}^{4}}}\; \P^{-1}\;
        \; \seminorm{g}_{\Cbspace^{3}\p{\vecZspace}} \eqperiod
    \end{aligned}
  \end{equation}
\end{theorem}
\begin{proof}
  Let \(a \equiv \p{a_{j}}_{j=1}^{d}\) for \(a_{j} \: : \: \rset^{d} \times \rset \to
  \rset\) and \(\Sigma \equiv \p{\Sigma_{j,j'}}_{j,j'=1}^{d} \defeq \sigma^{T} \sigma \) for \(\Sigma_{j,j'}
  \: : \: \rset^{d} \times \rset \to \rset\). For \(\vec x \equiv \p{\vec
    x_{i}}_{i=1}^{n} \equiv \p{x_{i,j}}_{i \in \br{1, \ldots, n},j \in \br{1, \ldots, d}}\),
  define the operators (recall that \(\mu_{t}\) is the law of \(Z\p{t}\))
  \[
    \def\uu{}
    \begin{aligned}
      \mathcal L_{\P} \uu
      &= \sum_{i=1}^{\P} \sum_{j=1}^{d} \sq[\Bigg]{\drift_{j}\p*{\vec x_{i}, \frac{1}{\P} \sum_{{i'=1}}^{\P} \kernel
        _{1}\p{\vec x_{i}, \vec x_{i'}}} \frac{\partial \uu}{\partial x_{i,j}} \\
      &\hskip 2cm+
        \frac{1}{2}  \sum_{j'=1}^{d}
        \Sigma_{j, j'}\p*{\vec x_{i},\frac{1}{\P} \sum_{{i'=1}}^{\P} \kernel_{2}\p{\vec x_{i}, \vec x_{i'}}}
        \frac{\partial^{2} \uu}{\partial x_{i,j} \partial x_{i,j'}}}
      \\\eqand
      \mathcal L_{\infty} \uu
      &= \sum_{i=1}^{\P} \sum_{j=1}^{d} \sq[\Bigg]{\drift_{j}\p*{\vec x_{i},\int_{\rset^{d}} \kernel_{1}\p{\vec x_{i}, \vec z} \law_{t}\p{\!\D \vec z}}
        \frac{\partial \uu}{\partial x_{i,j}}  \\
      &\hskip 2cm +
        \frac{1}{2} \sum_{j'=1}^{d}
        \Sigma_{j,j'}\p*{\vec x_{i},
        \int_{\rset^{d}} \kernel_{2}\p{\vec x_{i}, \vec z} \law_{t}\p{\!\D \vec z}}
        \frac{\partial^{2} \uu}{\partial x_{i,j} \partial x_{i,j'}}}\eqperiod
    \end{aligned}
  \]
  Consider the value function \(u\) satisfying the PDE
  \begin{equation}
    \label{eq:feynman-kac}
    \begin{aligned}
      &\frac{\partial u}{\partial t}\p{t,\vec x} + \mathcal L_{\P} u\p{t,\vec x} = 0, ~\text{for }
        0\leq t<T \text{ and } \vec x \in \vecZspace \\
      \eqand& u\p{T, \vec x} = g\p{\vec x}~\text{for }
              \vec x \in \vecZspace \eqperiod
    \end{aligned}
  \end{equation}
  Under \cref{eq:C3-boundedness}, \changed{a strong solution for
    \cref{eq:X-sys} exists and is unique \cite[Theorem~1.1 in
    Chapter~5]{friedman:sde-book} and} \( u\p{t, \vec x} = \E*{g\p{\vec
      X^{\P}\p{T}} \given \vec X^{\P}\p{t} = \vec x}\), see \cite[Theorem~6.1
  in Chapter~5]{friedman:sde-book}.
  \changed{Given the existence of a solution to \cref{eq:Z-MV} and its law,
    and recalling that the coefficients \(\drift\) and \(\diffusion\) in
    \cref{eq:Z-MV} are integrable and square-integrable, respectively, due to
    \cref{eq:C3-boundedness} and \cref{eq:Z-moments}, we define \(U\p{t}
    \defeq {u}\p{t, \vec Z^{\P}\p{t}}\) and apply It\^o's formula
    \changed{\cite[Theorem~5.3 in Chapter~4]{friedman:sde-book}} to arrive at}
  \begin{equation}\label{eq:weak-conv-representation}
    \begin{aligned}
      \E*{g\p*{\vec Z^{\P}\p{T}}} - \E*{g\p*{\vec X^{\P}\p{T}}}
      &= \E{U\p{T} - U\p{0}} \\
      &\begin{details}
        = \E*{ \int_{0}^{T} \D U\p{t}}\\
        & = \E*{ \int_{0}^{T} \p*{\mathcal L_{\infty} - \mathcal L_{\P}}u\p{t, \vec Z^{\P}\p{t}} \D t}\\
        &
      \end{details}
       = \int_{0}^{T}\E*{ \p*{\mathcal L_{\infty} - \mathcal L_{\P}}u\p{t, \vec Z^{\P}\p{t}} }\D t\eqperiod
    \end{aligned}
  \end{equation}
  The last equality is satisfied under the boundedness conditions in
  \cref{eq:C3-boundedness}, the integrability of \(Z\), and the boundedness of
  the derivatives of \(u\), which we will establish later. Then
  \begin{equation}\label{eq:Linfy-Ld}
    \begin{aligned}
      \E*{\p*{\mathcal L_{\infty} - \mathcal L_{\P}}u\p{t, \vec Z^{n}\p{t}}}
      &= \sum_{i=1}^{\P} \sum_{j=1}^{d}
        \E*{\Delta_{i} \drift_{j} \, \frac{\partial u}{\partial x_{i,j}}\p{t, \vec Z^{{n}}\p{t}}} \\
      &\hskip 1.5cm + \frac{1}{2} \sum_{j'=1}^{d} \E*{ \Delta_{i} \Sigma_{j,j'} \, \frac{\partial^{2} u}{\partial x_{i,j} \partial x_{i,j'}}\p{t, \vec Z^{{n}}\p{t}}}\eqcomma
    \end{aligned}
  \end{equation}
  where for \(f \equiv a_{j}, \kappa \equiv \kappa_{1}\) and \(f\equiv \Sigma_{j,j'}, \kappa\equiv \kappa_{2}\), we define
  \begin{equation}\label{eq:f-diff}
    \Delta_{i} f \defeq f\p*{Z_{i}\p{t}, \int_{\rset^{d}} \kernel\p{Z_{i}\p{t}, \vec z} \law_{t}\p{\!\D \vec z}}
    - f\p*{Z_{i}\p{t}, \frac{1}{\P}\sum_{j=1}^{\P} \kernel\p{Z_{i}\p{t}, Z_{j}\p{t}}} \eqperiod
  \end{equation}
  Using the triangle inequality, \cref{lem:weak-Z-rate} and
  \cref{eq:Z-moments}, we bound
  \[
    \begin{aligned}
      &\abs{\E*{\p*{\mathcal L_{\infty} - \mathcal L_{\P}}u\p{t, \vec Z^{n}\p{t}}}}\\
      &\lesssim
        \P^{-1} \: \p*{1+\E{\norm{Z\p{t}}^{4}}} \,
        \p*{\sum_{i=1}^{n}\sum_{j=1}^{d}
        \p*{\norm*{\frac{\partial u\p{t, \cdot}}{\partial x_{i,j}}}_{C^{1}\p{\vecZspace}}
        + \sum_{j'=1}^{d} \norm*{\frac{\partial^{2} u\p{t, \cdot}}{\partial x_{i,j} \partial x_{i,j'}}}_{C^{1}\p{\vecZspace}}} } \\
      &\lesssim
        \P^{-1} \p{1+ \E{\norm{\xi}^{4}}} \: \seminorm{u\p{t,\cdot}}_{\Cbspace^{3}\p{\vecZspace}}\eqperiod
    \end{aligned}
  \]

  It remains to show %
  the bound \(\seminorm{u\p{t,\cdot}}_{\Cbspace^{3}\p{\vecZspace}} \lesssim
  \seminorm{g}_{\Cbspace^{3}\p{\vecZspace}}\) for \(t \leq T\). %
  To that end, we use \cref{lemma:FK-Dell-bound} in the following section with
  an appropriate definition of \(\gdrift\) and \(\gdiffusion\) in terms of
  \(\drift\) and \(\diffusion\), respectively, since
  \cref{ass:sde-bounded-derv} is satisfied for \(q=3\) given
  \cref{eq:C3-boundedness}; see the discussion after
  \cref{ass:sde-bounded-derv}.
\end{proof}

\begin{corollary}\label{thm:finite-k}
  From the previous theorem, we can readily deduce that under the same
  conditions and for an integer \(k \leq \P\) and \(g: \rset^{k \times d} \to \rset\),
  we have
  \[
    \begin{aligned}
      &\abs*{\E{g\p{\X{1}{\P}\p{T}, \X{2}{\P}\p{T}, \ldots, \X{k}{\P}\p{T}} - g\p{Z_{1}\p{T}, Z_{2}\p{T}, \ldots, Z_{k}\p{T}} }} \\
      &\pushright{\lesssim
        \frac{1}{\P} \:  \p{1 + \E{\norm{\xi}^{4}}} \:
        \binom{d\, k+2}{3}
        \max_{\vec \ell \in \nset^{d k}, 1 \leq \abs{\vec \ell} \leq 3} \p*{\norm*{\frac{\partial^{\abs{\vec \ell}} g}{\partial \vec x^{\vec \ell}}}_{C\p{\rset^{k \times d}}}}\eqperiod}
    \end{aligned}
  \]
\end{corollary}
\begin{details}
  \begin{proof}
    The result follows from the inequality
    \[
      \begin{aligned}
        \norm{g}_{\Cbspace^{3} \p{\rset^{k \times d}}}
        &= \sum_{\vec \ell \in \nset^{d k}, 1 \leq
          \abs{\vec \ell} \leq 3} \norm*{\frac{\partial^{\abs{\vec \ell}} g}{\partial \vec x^{\vec \ell}
          }}_{C\p{\rset^{k \times d}}}\\
        &\leq
          \p*{\max_{\vec \ell \in \nset^{d k}, 1 \leq
          \abs{\vec \ell} \leq 3}\norm*{\frac{\partial^{\abs{\vec \ell}} g}{\partial \vec x^{\vec \ell}
          }}_{C\p{\vecZspace}}}
          \sum_{\vec \ell \in \nset^{k}, 1 \leq
          \abs{\vec \ell} \leq 3} 1
      \end{aligned}
    \]
    and since
    \[
      \sum_{\vec \ell \in \nset^{d k}, 1 \leq \abs{\vec \ell} \leq 3} 1 =
      3 \binom{d\,k}{1} %
      + 3 \binom{d\,k}{2} %
      + \binom{d\,k}{3}   %
      \leq 3 \binom{d\,k+2}{3}
    \]
  \end{proof}
\end{details}
\cref{thm:finite-k} is useful when \(k\) is independent of \(\P\). For
example, consider for \(\tilde g : \rset^{d} \to \rset\),
\[
  g\p{\vec X^{\P}\p{T}} \defeq \frac{1}{\P} \sum_{i=1}^{\P} \tilde g\p{\X{i}{\P}\p{T}}\eqcomma
\]
and note that \( \seminorm {g}_{\Cbspace^{3}\p{\vecZspace}} = \seminorm {\tilde
  g}_{\Cbspace^{3}\p{\rset^{d}}} \) and hence \cref{thm:weak-conv} implies that
\[
  \abs*{\E*{g\p{\vec X^{\P}\p{T}} - g\p{\vec Z^{n}\p{T}}}} \lesssim \P^{-1}\; \;
  \seminorm{\tilde g}_{\Cbspace^{3}\p{\rset}} = \Order{\P^{-1}} \eqperiod
\]
We would not get the same result if we apply \cref{thm:finite-k} directly to \(g\) with
\(k=\P\). Instead, in this particular case, we may apply \cref{thm:finite-k} to \(\tilde g\)
with \(k=1\) to see that \(\E*{\tilde g\p{\X{1}{\P}\p{T}} - \tilde g\p{Z\p{T}}} =
\Order{\P^{-1}}\) and conclude by noting that
\[
  \E*{g\p{\vec X^{\P}\p{T}} - g\p{\vec Z^{n}\p{T}}} = \E*{\tilde
    g\p{\X{1}{\P}\p{T}} - \tilde g\p{Z\p{T}}} \eqperiod
\]

\begin{comment}
  The number of terms in \(\br{ \vec \ell \in \nset^{k} : 0 \leq \abs{\vec \ell} \le p}\) is
  \[\binom{p+k}{k}=\binom{p+k}{p}\eqperiod\]
  We prove this by induction on \(k\). For \(k=1\), the total number is
  exactly \(p+1\). Consider \(k+1\) there are, the total number is
  \[
    \sum_{i=0}^{p} \sum_{\br{ \vec \ell \in \nset^{k} : 0 \leq \abs{\vec \ell} \le p-i}} 1 = \sum_{i=0}^{p}\binom{p+k-i}{k}
    = \sum_{i=0}^{p} \binom{i+k}{k}
    = \sum_{i=0}^{p} \binom{i+k}{i}
    =  \binom{p+k+1}{k+1}
  \]
  %
\end{comment}

\begin{remark}
  In the special case when \(\kernel_{2}=0\), we can relax \cref{eq:C3-boundedness} and only
  assume that
  \begin{equation}\label{eq:C2-boundedness}
    \begin{aligned}
      \seminorm{\drift}_{\Cbspace^{2}\p{\rset^{d} \times \rset  ; \rset^{d} }}
      + \seminorm{\diffusion}_{\Cbspace^{2}\p{\rset^{d}\times \rset  ; \rset^{d \times d'}}}
      + \seminorm{\kernel_{1}}_{\Cbspace^{2}\p{\rset^{d}\times \rset^{d}}} < \infty
      \eqperiod%
    \end{aligned}
  \end{equation}
  The result would then also involve only the first and second derivatives of
  \(g\),
  \begin{equation*}
    \begin{aligned}
      \abs*{\E*{g\p{\vec X^{\P}\p{T}} - g\p{\vec Z^{n}\p{T}}}}
      &\lesssim  \P^{-1}\; \p{1 + \E{\norm{\xi}^{4}}} \;
                                                             \seminorm{g}_{\Cbspace^{2}\p{\vecZspace}} \eqcomma
    \end{aligned}
  \end{equation*}
  thus recovering a similar result to the one obtained, for example, in \cite[Chapter
  9]{kolokoltsov:book-nonlinear}.
\end{remark}

\begin{remark}[Multi-dimensional interaction kernels]
  \label{rem:multi-kernels}
  The result of \cref{thm:weak-conv} can be extended to multi-dimensional
  kernels, \(\kernel_{1}, \kernel_{2} : \rset^{d} \times \rset^{d} \to \rset^{m}\),
  for some integer \(m\), assuming that
  \[
    \seminorm{\kernel_{1}}_{\Cbspace^{3}\p{\rset^{d}\times \rset^{d}\, ;\,
      \rset^{m} }} + \seminorm{\kernel_{2}}_{\Cbspace^{3}\p{\rset^{d}\times
        \rset^{d}\, ;\, \rset^{m} }} < \infty\eqperiod
  \]
  The proof would follow the same steps by adding and subtracting appropriate
  terms in \eqref{eq:f-diff} and treating each component separately.
\end{remark}

\section{Moments Bounds for SDE Variations with Sobolev-Bounded
  Coefficients}\label{sec:sde-variation-bounds}
\def\Linftysp{\bddspace}
\def\bddspace{\infty}
\def\dim{\P}
In this section, for $T>0$ and $\p{t,\vec x} \in [0,T] \times \rset^{\dim}$, we consider a general SDE of
the form
\begin{equation}\label{eq:sde}
  \vec X^{t, \vec x}\p{s} = \vec x + \int_t^s\vec \gdrift\p{\tau,\vec X^{t, \vec x}\p{\tau}} \D \tau
  +  \int_t^s \vec \gdiffusion\p{\tau, \vec X^{t, \vec x}\p{\tau}} \D \vec W\p{\tau}, \quad s \in [t,T]\eqcomma
\end{equation}
with drift coefficient $\vec \gdrift \equiv \p{\gdrift_{i}}_{i=1}^{\dim} : [0,T] \times
\rset^{\dim} \to \rset^{\dim}$, a diffusion coefficient $\vec \gdiffusion \equiv
\p{\gdiffusion_{i,m}}_{i \in \br{1, \ldots, \dim}, m \in \br{1, \ldots, \dim'}}: [0,T] \times
\rset^{\dim} \to \rset^{\dim \times \dim'}$, and $\vec W%
$ is a vector of \(\dim'\) independent standard Wiener processes over a
probability space \(\p{\Omega, \mathcal F, \mathbb P}\) with the natural
filtration. {The main results of this section are
  \cref{prop:sde-variation,lemma:FK-Dell-bound} with the latter being used in
  the final step of the proof of \cref{thm:weak-conv}. Note that the system
  that we consider in \cref{thm:weak-conv} is \eqref{eq:X-sys} which is a
  specific example of \eqref{eq:sde} with \(\dim \leftarrow \P d \). However, we prove
  the results more generally for \eqref{eq:sde} to emphasize that the
  particular structure of \eqref{eq:X-sys} is irrelevant as long as
  \cref{ass:sde-bounded-derv} is satisfied. Additionally, the results in this
  section could be useful beyond the current work.}
In what follows, for any
\(f : [0,T] \times \rset^{\dim} \to \rset\), define the \changed{extended} norm
\[
  \norm{f}_{\infty} \defeq \norm{f}_{L^{\infty}\p{[0,T]; {C\p{\rset^{\dim}}}}} \defeq
  \sup_{0 \leq s \le T} \norm{f\p{s, \cdot}}_{C\p{\rset^{\dim}}}\eqperiod
\]
For brevity of presentation, we will define for the coefficients \(\gdrift\)
and \(\gdiffusion\) in \eqref{eq:sde} and for an integer \(r \geq 0\),
\begin{equation}\label{eq:vector-Dq-def}
  \norm{\gdrift}_{\mathrm D^{r}, \infty} \defeq \norm{\partial \gdrift}_{\mathrm D^{r-1}, \infty} +
  \sum_{\vec \ell \in \nset^{\dim}, \abs{\vec \ell} = r} \max_{\substack{i \in \br{1, \ldots, \dim} \\
      \textnormal{s.t. } \ell_{i}=0}} \, \norm*{\frac{\partial^r \gdrift_{i}}{\partial \vec x^{\vec \ell}}}_{\Linftysp}
\end{equation}
where \(\partial \gdrift = \p*{\frac{\partial \gdrift_{i}}{\partial x_{i}}}_{i=1}^{\dim}\) and
\(\norm{\partial \gdrift}_{\mathrm D^{-1}, \Linftysp} \defeq 0\). Similarity, we define
\begin{equation}\label{eq:matrix-Dq-def}
  \norm{\gdiffusion}_{\mathrm D^{r}, \ell^{2}_{\Linftysp}} \defeq \norm{\partial
    \gdiffusion}_{\mathrm D^{r-1}, \ell^{2}_{\Linftysp}} +
  \sum_{\vec \ell \in \nset^{\dim}, \abs{\vec \ell} = r} \max_{\substack{i \in \br{1, \ldots, \dim} \\
      \textnormal{s.t. } \ell_{i}=0}} \, \p*{\sum_{m=1}^{\dim'} \norm*{\frac{\partial^r \gdiffusion_{im}}{\partial
        \vec x^{\vec \ell}}}_{\Linftysp}^{2}}^{1/2}
\end{equation}
where \(\partial \gdiffusion = \p*{\frac{\partial \gdiffusion_{i m}}{\partial x_{i}}}_{i=\br{1,\ldots,
    \dim}, m=\br{1,\ldots, \dim'}}\) and \(\norm{\gdiffusion}_{\mathrm D^{-1},
  \ell^{2}_{\Linftysp}} \defeq 0\). Finally, for the process \(\vec X^{t, \vec
  x}\) in \eqref{eq:sde}, we define for any \(p \geq 1\),
\begin{equation}\label{eq:process-Dq-def}
  \begin{aligned}
    \norm{\vec X^{t, \vec x}}_{\mathrm D^{r}, L^{\infty}\p{[t, T]; L^{p}\p{\Omega, \mathbb P}}}
    &\defeq
      \norm{\partial \vec X^{t, \vec x}}_ {\mathrm D^{r-1}, L^{\infty}\p{[t, T]; L^{p}\p{\Omega, \mathbb P}}} \\
    &\hskip 1cm+
      \sum_{\vec \ell \in \nset^{\dim}, \abs{\vec \ell} = r} \max_{\substack{i \in \br{1, \ldots, \dim} \\
    \textnormal{s.t. } \ell_{i}=0}} \, \sup_{t \leq s \leq T}
    \E*{\abs*{\frac{\partial^{r} X_{i}^{t, \vec x}}{\partial \vec x^{\vec \ell}}\p{s}}^{p}}^{1/p}
  \end{aligned}
\end{equation}
where \(\partial \vec X^{t, \vec x} = \p*{\frac{\partial X_{i}^{t, \vec x}}{\partial
    x_{i}}}_{i=1}^{\dim}\) and \( \norm{\vec X^{t, \vec x}}_
{\mathrm D^{-1}, L^{\infty}\p{[t, T]; L^{p}\p{\Omega, \mathbb P}}} \defeq 0\). %

\begin{assumption}[Bounded derivatives]\label{ass:sde-bounded-derv}
  For an integer \(q\) we assume that $\vec \gdrift : [0,T] \times \rset^{\dim} \to
  \rset^{\dim}$ and $\vec \gdiffusion: [0,T] \times \rset^{\dim} \to \rset^{\dim \times
    \dim'}$ satisfy
    \[\sum_{r=1}^{q} \norm{\gdrift}_{\mathrm D^{r}, \Linftysp} +
      \norm{\gdiffusion}_{\mathrm D^{r}, \ell^{2}_{\Linftysp}} < C_{q}
      \eqcomma \] for some constant \(C_q > 0\) independent of \(\dim\).
\end{assumption}

The previous assumption deserves some explanation.
For example, focusing on the drift coefficient \(\gdrift\), the definition in
\cref{eq:vector-Dq-def} %
for \(r=1\) simplifies to
\[
  \begin{aligned}
    \norm{\gdrift}_{\mathrm D^{1}, \Linftysp} =
    \max_{i} \norm*{ \pderv*{\gdrift_{i}}{x_{i}}}_{\Linftysp} +
    {\sum_{\ell=1}^{\dim}
    \max_{i \neq \ell}\norm*{\pderv*{\gdrift_{i}}{x_{\ell} }}_{\Linftysp}}
    \eqcomma
  \end{aligned}
\]
and \(\norm{\gdiffusion}_{\mathrm D^{1}, \ell^{2}_{\Linftysp}}\) expands
similarly. Hence, a sufficient condition for \cref{ass:sde-bounded-derv} when
\(q=1\) is to bound
\[
  \p*{\sum_{m=1}^{\dim'}\norm*{\frac{\partial \gdiffusion_{im}}{\partial x_{\ell}}}_{\Linftysp}^{2}}^{1/2} +
  \norm*{\pderv*{\gdrift_{i}}{x_{\ell} }}_{\Linftysp} \leq \begin{cases}
                                                       \tilde C_{1} & i=\ell\eqcomma\\
                                                       \tilde C_1 \, \dim^{-1} & i \neq \ell\eqcomma\\
  \end{cases}
\]
for all \(i \in \br{1, \ldots, \dim}\) and some constant \(\tilde C_{1}>0\).
For \(r=2\), the definition \cref{eq:vector-Dq-def} simplifies to
\[
  \norm{\gdrift}_{\mathrm D^{2}, \Linftysp} =
  \max_{i} \norm*{\frac{\partial^{2} \gdrift_{i}}{ \partial x_{i}^{2} }}_{\Linftysp} + \sum_{\ell=1}^{\dim} \max_{i \neq \ell}\norm*{\frac{\partial^{2} \gdrift_{i}}{ \partial x_{i} \partial x_{\ell} }}_{\Linftysp} + \sum_{\ell=1}^{\dim} \sum_{\ell'=1}^{\dim}\max_{i \notin \br{\ell, \ell'}} \norm*{\frac{\partial^{2} \gdrift_{i}}{ \partial x_{\ell} \partial x_{\ell'} }}_{\Linftysp}\eqcomma
\]
and an additional condition on the second derivatives would be required
\cref{ass:sde-bounded-derv} when \(q=2\), for example
\[
  \p*{\sum_{m=1}^{\dim'}\norm*{\frac{\partial^r \gdiffusion_{im}}{\partial x^{\vec \ell}}}_{\Linftysp}^{2}}^{1/2}
  + \norm*{\frac{\partial^{2} \gdrift_{i}}{ \partial x_{\ell} \partial x_{\ell'} }}_{\Linftysp}
  \leq \begin{cases}
      \tilde C_{2} & i= \ell= \ell'  \\
      \tilde C_2 \, {\dim^{-1}} & i=\ell\neq\ell' \text{ or } i\neq\ell=\ell'\\
      \tilde C_2 \, {\dim^{-2}} & i, \ell, \ell' \text{ are distinct,}
  \end{cases}
\]
for all \(i \in \br{1, \ldots, \dim}\) and some constant \(\tilde C_{2} >0\). In
general, for any integer \(q>0\), a sufficient condition for
\cref{ass:sde-bounded-derv} is
\[
  \p*{\sum_{m=1}^{\dim'} \norm*{\frac{\partial^{\abs{\vec \ell}} \gdiffusion_{im}}{ \partial \vec
        x^{\vec \ell} }}_{\Linftysp}^{2}}^{1/2} + \norm*{\frac{\partial^{\abs{\vec \ell}}
      \gdrift_{i}}{ \partial \vec x^{\vec \ell} }}_{\Linftysp} \leq \tilde C_{q} \,
  {\dim^{1-\abs{\vec \ell+\vec e_{i}}_{0}}} \quad \text{for all } i \in \br{1,\ldots,\dim}
  \text{ and } \vec \ell \in \nset^{\dim} \::\: \abs{\vec \ell} \le q\eqcomma
\]
for a constant \(\tilde C_{q}>0\) and where \(\vec e_{i}\) is the \(\nth{i}\)
unit vector and \(\abs{\vec \ell}_{0}\) denotes the number of non-zero elements
of \(\vec \ell\).
\begin{lemma}[\(L^p\) bound of stochastic flows]\label{prop:sde-variation}
  Let \cref{ass:sde-bounded-derv} be satisfied for \(q \in \br{1, 2, 3}\) and
  for \(\vec X^{t, \vec x}\) in \cref{eq:sde} and any \(p \geq 2\) then there
  exists constants \(K_{q, p}\), independent of \(\dim\) and \(\vec x\), such
  that
    \[
      \norm{\vec X^{t, \vec x}}_{\mathrm D^{q}, L^{\infty}\p{[t, T]; L^{p}\p{\Omega,
            \mathbb P}}} \leq K_{q, p}\eqperiod
    \]
\end{lemma}
\begin{proof}
  We first note the following inequality for any index sets \(\mathcal I\) and
  \(\mathcal J\) and any sequence \(\br{a_{i,j}}_{i \in \mathcal I, j \in \mathcal
    J}\),
  \begin{equation}\label{eq:jensen-ineq-seq}
    \p*{\sum_{j\in \mathcal J} \p*{\sum_{i \in \mathcal I} a_{i,j}}^{2}}^{1/2}
    \leq \sum_{i \in \mathcal I} \p*{\sum_{j\in \mathcal J} a_{i,j}^{2}}^{1/2}\eqcomma
  \end{equation}
  which can be shown by expanding the square and using \holdereq{}.
  \begin{details}
    \[
      \begin{aligned}
        \p*{\sum_{j\in \mathcal J} \p*{\sum_{i \in \mathcal I} a_{i,j}}^{2}}^{1/2}
        &= \p*{ \sum_{i \in \mathcal I} \sum_{i' \in \mathcal I} \sum_{j \in \mathcal J} a_{i,j}
          a_{i',j}}^{1/2}\\
        &\leq \p*{ \sum_{i \in \mathcal I} \sum_{i' \in \mathcal I}
          \p*{\sum_{j \in \mathcal J} a_{i,j}^{2}}^{1/2}
          \p*{\sum_{j \in \mathcal J}  a_{i',j}^{2}}^{1/2}}^{1/2}\\
        &= \sum_{i \in \mathcal I} \p*{\sum_{j\in \mathcal J} a_{i,j}^{2}}^{1/2}\eqperiod
      \end{aligned}
    \]
  \end{details}
  Using \eqref{eq:jensen-ineq-seq} and Jensen's inequality, we can show the
  following inequality for any random variables \(\p{Y_{i}}_{i \in \mathcal I}\)
  and measurable sequence \(\br{a_{i,j}}_{i \in \mathcal I, j \in \mathcal J}\)
  \begin{equation}\label{eq:jensen-ineq}
    \begin{aligned}
      \E*{\p*{\sum_{j \in \mathcal J} \p*{\sum_{i \in \mathcal I} a_{i,j} Y_{i}}^{2}}^{p/2}}
      &\leq \p*{\sum_{i \in \mathcal I} \p*{\sum_{j\in \mathcal J} a_{i,j}^{2}}^{1/2}}^{p} \max_{i \in \mathcal I} \E{ \abs{Y_{i}}^{p} }\eqperiod
    \end{aligned}
  \end{equation}
  \begin{details}
    \begin{align*}
      \E*{\p*{\sum_{j \in \mathcal J} \p*{\sum_{i \in \mathcal I} a_{i,j} Y_{i}}^{2}}^{p/2}}
      &\leq \E*{\p*{\sum_{i \in \mathcal I} \p*{\sum_{j\in \mathcal J} a_{i,j}^{2} }^{1/2} \abs{Y_{i}}}^{p}} \\
      &\leq{\p*{\sum_{i \in \mathcal I} \p*{\sum_{j\in \mathcal J} a_{i,j}^{2} }^{1/2}}^{p-1}}
        \E*{\p*{\sum_{i \in \mathcal I} \p*{\sum_{j\in \mathcal J} a_{i,j}^{2} }^{1/2} \abs{Y_{i}}^{p}}} \eqperiod
    \end{align*}
  \end{details}
  In addition, note that for a positive sequence \(\br{a_{i,j}}_{i \in \mathcal
    I, j \in \mathcal J}\), we have
  \begin{equation}\label{eq:max-seq}
    \max_{i \in \mathcal I}\sum_{j \in \mathcal J} a_{i,j} \leq
    \max_{i \in \mathcal I}a_{i,i} + \max_{i \in \mathcal I} \sum_{j \in \mathcal J,i\neq j}^{\dim} a_{i,j} \leq
    \max_{i \in \mathcal I}a_{i,i} +  \sum_{j \in \mathcal J} \max_{i \in \mathcal I,j\neq i} a_{i,j}
  \end{equation}
  \subparagraph{First variation} First, note that the process
  \(\pderv*{X^{t,\vec x}_{i}}{x_{j}}\) exists under
  \cref{ass:sde-bounded-derv} for \(q=1\) and satisfies for \(s \geq t\) the SDE
  \[
    \begin{aligned}
      \pderv*{X^{t,\vec x}_{i}}{x_{j}}\p{s} = \delta_{i,j} &+ \int_t^s \sum_{k=1}^{\dim} \pderv*{\gdrift_{i}}{x_{k}}\p{\tau,\vec X^{t,\vec x}\p{\tau}} \, \pderv*{X^{t,\vec x}_{k}}{x_{j}}\p{\tau} \, \D \tau \\
      &+ \sum_{m=1}^{\dim'} \int_t^s \sum_{k=1}^{\dim}  \pderv*{\gdiffusion_{im}}{x_{k}} \p{\tau, \vec X^{t,\vec x}\p{\tau}} \pderv*{X^{t,\vec x}_{k}}{x_{j}} \p{\tau} \, \D W_{m}\p{\tau}\eqcomma
    \end{aligned}
  \]
  cf. \cite[Theorem 5.3 in Chapter 5]{friedman:sde-book}. %
  By \ito's formula
  \begin{equation}\label{eq:X-first-var-p-moment}
    \begin{aligned}
      &\E*{\abs*{\pderv*{X^{t,\vec x}_{i}}{x_{j}}\p{s}}^{p}} \leq \delta_{i,j} + p \int_{t}^{s} \E*{\abs*{\sum_{k=1}^{\dim} \pderv*{\gdrift_{i}}{x_{k}}\p{\tau,\vec X^{t,\vec x}\p{\tau}} \, \pderv*{X^{t,\vec x}_{k}}{x_{j}}\p{\tau} } \: \abs*{\pderv*{X^{t,\vec x}_{i}}{x_{j}}\p{\tau}}^{p-1}} \,
      \D \tau \\
      &\hskip 1cm + \frac{p\p{p-1}}{2} \int_{t}^{s} \sum_{m=1}^{\dim'} \E*{ \p*{\sum_{k=1}^{\dim} \pderv*{\gdiffusion_{im}}{x_{k}} \p{\tau, \vec X^{t,\vec x}\p{\tau}} \pderv*{X^{t,\vec x}_{k}}{x_{j}} \p{\tau}}^{2}
        \abs*{\pderv*{X^{t,\vec x}_{i}}{x_{j}}\p{\tau}}^{p-2}} \, \D \tau
      \eqperiod
    \end{aligned}
  \end{equation}
  For the term involving \(\gdrift_{i}\), using Young's inequality, we can bound
  \begin{equation}
    \label{eq:X-first-var-p-moment-term-1}
    \begin{aligned}
      & \E*{\abs*{\sum_{k=1}^{\dim} \pderv*{\gdrift_{i}}{x_{k}}\p{\tau,\vec X^{t,\vec x}\p{\tau}} \, \pderv*{X^{t,\vec x}_{k}}{x_{j}}\p{\tau} } \: \abs*{\pderv*{X^{t,\vec x}_{i}}{x_{j}}\p{\tau}}^{p-1}} \\
      &\leq
        \E*{
        \p*{\sum_{k=1}^{\dim} \norm*{\pderv*{\gdrift_{i}}{x_{k}}}_{\bddspace}\, \abs*{\pderv*{X^{t,\vec x}_{k}}{x_{j}}\p{\tau} }} \: \abs*{\pderv*{X^{t,\vec x}_{i}}{x_{j}}\p{\tau}}^{p-1}} \,\\
      &\leq \frac{1}{p} \, \E*{\
        \p*{\sum_{k=1}^{\dim}  \norm*{\pderv*{\gdrift_{i}}{x_{k}}}_{\bddspace} \abs*{\pderv*{X^{t,\vec x}_{k}}{x_{j}}\p{\tau} }}^{p}} + \frac{p-1}{p}\E*{\abs*{\pderv*{X^{t,\vec x}_{i}}{x_{j}}\p{\tau}}^{p}} \,
        \eqperiod
    \end{aligned}
  \end{equation}
  By using Jensen's inequality, \cref{eq:jensen-ineq} and
  \cref{ass:sde-bounded-derv}, we can further bound
  \begin{equation}\label{eq:X-first-var-p-moment-term-1-1}
    \begin{aligned}
      &\E*{\
        \p*{\sum_{k=1}^{\dim}  \norm*{\pderv*{\gdrift_{i}}{x_{k}}}_{\bddspace} \abs*{\pderv*{X^{t,\vec x}_{k}}{x_{j}}\p{\tau} }}^{p}}\\
      &\leq { 2^{p-1}} \norm*{\pderv*{\gdrift_{i}}{x_{j}}}_{\bddspace}^{p} \E*{\abs*{\pderv*{X^{t,\vec x}_{j}}{x_{j}}\p{\tau} }^{p}}+
      2^{p-1} \E*{\
        \p*{\sum_{k=1, k\neq j}^{\dim}  \norm*{\pderv*{\gdrift_{i}}{x_{k}}}_{\bddspace} \abs*{\pderv*{X^{t,\vec x}_{k}}{x_{j}}\p{\tau} }}^{p}} \\
      &\leq {2^{p-1}} \, \norm*{\pderv*{\gdrift_{i}}{x_{j}}}_{\bddspace}^{p} \, \E*{\abs*{\pderv*{X^{t,\vec x}_{j}}{x_{j}}\p{\tau} }^{p}}+ {2^{p-1} C_{1}^{p}} \, \max_{k \neq j}\E*{\abs*{\pderv*{X^{t,\vec x}_{k}}{x_{j}}\p{\tau} }^{p}}\eqperiod
    \end{aligned}
  \end{equation}
  For the term involving \(\gdiffusion_{i\cdot}\), using Young's inequality and
  bounding the derivatives of \(\gdiffusion_{i\cdot}\),
  \begin{equation}\label{eq:X-first-var-p-moment-term-2}
    \begin{aligned}
      &\E*{ \p*{\sum_{k=1}^{\dim} \pderv*{\gdiffusion_{im}}{x_{k}} \p{\tau, \vec X^{t,\vec x}\p{\tau}} \pderv*{X^{t,\vec x}_{k}}{x_{j}} \p{\tau}}^{2}
        \abs*{\pderv*{X^{t,\vec x}_{i}}{x_{j}}\p{\tau}}^{p-2}} \\
      &\leq \frac{2}{p} \E*{ \p*{\sum_{m=1}^{\dim'} \p*{\sum_{k=1}^{\dim} \norm*{\pderv*{\gdiffusion_{im}}{x_{k}} }_{\bddspace} \: \abs*{\pderv*{X^{t,\vec x}_{k}}{x_{j}} \p{\tau}}}^{2}}^{p/2}}
      + \frac{p-2}{p}\E*{\abs*{\pderv*{X^{t,\vec x}_{i}}{x_{j}}\p{\tau}}^{p}}\eqperiod
    \end{aligned}
  \end{equation}
  Using Jensen's inequality and~\cref{eq:jensen-ineq} and
  \cref{ass:sde-bounded-derv}, we can further bound
  \begin{equation}
    \label{eq:X-first-var-p-moment-term-2-1}
    \begin{aligned}
      &\E*{
        \p*{\sum_{m=1}^{\dim'} \p*{\sum_{k=1}^{\dim} \norm*{\pderv*{\gdiffusion_{im}}{x_{k}} }_{\bddspace} \: \abs*{\pderv*{X^{t,\vec x}_{k}}{x_{j}} \p{\tau}}}^{2}}^{p/2}}\\
      &\leq 2^{p-1}\: \p*{\sum_{m=1}^{\dim'} \norm*{\pderv*{\gdiffusion_{im}}{x_{j}} }_{\bddspace}^{2}}^{p/2} \:
        \E*{\abs*{\pderv*{X^{t,\vec x}_{j}}{x_{j}} \p{\tau}}^{p}}\\
      &\hskip 1cm +2^{p-1}\E*{\p*{\sum_{m=1}^{\dim'} \p*{\sum_{k=1, k\neq j}^{\dim} \norm*{\pderv*{\gdiffusion_{im}}{x_{k}} }_{\bddspace} \: \abs*{\pderv*{X^{t,\vec x}_{k}}{x_{j}} \p{\tau}}}^{2} }^{p/2}} \\
      &\leq
        2^{p-1}\: \p*{\sum_{m=1}^{\dim'} \norm*{\pderv*{\gdiffusion_{im}}{x_{j}} }_{\bddspace}^{2}}^{p/2} \:
        \E*{\abs*{\pderv*{X^{t,\vec x}_{j}}{x_{j}} \p{\tau}}^{p}} + 2^{p-1} C_{1}^{p}\:
        \max_{k \neq j} \E*{\abs*{\pderv*{X^{t,\vec x}_{k}}{x_{j}} \p{\tau}}^{p}}\eqperiod
    \end{aligned}
  \end{equation}
  Combining \cref{eq:X-first-var-p-moment-term-1-1} with
  \cref{eq:X-first-var-p-moment-term-1} and
  \cref{eq:X-first-var-p-moment-term-2-1} with
  \cref{eq:X-first-var-p-moment-term-2} and substituting into
  \cref{eq:X-first-var-p-moment}
  and simplifying yield
  \begin{equation}\label{eq:final-1st-var}
    \begin{aligned}
      \E*{\abs*{\pderv*{X^{t,\vec x}_{i}}{x_{j}}\p{s}}^{p}} \leq&& \delta_{ij}+&
      c_{1} \, \int_{t}^{s}  \E*{\abs*{\pderv*{X^{t,\vec x}_{i}}{x_{j}}\p{\tau}}^{p}} \D \tau \\
      &&+& c_{2} \, C_{1}^{p} \, \int_{t}^{s} \max_{k \ne j} \E*{ \abs*{\pderv*{X^{t,\vec x}_{k}}{x_{j}}\p{\tau} }^{p}} \D \tau\\
      &&+& c_{2} \, a_{i,j}^{p} \int_{t}^{s} \E*{ \abs*{\pderv*{X^{t,\vec x}_{j}}{x_{j}}\p{\tau} }^{p}}
      \: \D \tau\eqcomma
       \end{aligned}
  \end{equation}
where
  \[
    \begin{aligned}
      c_{1} &\defeq \frac{p \p{p-1}}{2}\eqcomma\\
      c_{2} &\defeq p \, 2^{p-1} \\
      \eqand a_{i,j}^{p} &\defeq \norm*{\pderv*{\gdrift_{i}}{x_{j}}}_{\bddspace}^{p} \, +  \p*{\sum_{m=1}^{\dim'} \norm*{\pderv*{\gdiffusion_{im}}{x_{j}} }_{\bddspace}^{2}}^{p/2} \:\eqcomma
    \end{aligned}
  \]
  so that \(\max_{i} a_{i,i} + \sum_{j=1}^{\dim} \max_{i \ne j} a_{i,j} \leq C_{1}\)
  by \cref{ass:sde-bounded-derv}. Then, taking the maximum over all \(i\) in
  \cref{eq:final-1st-var}, we arrive at
  \[
    \max_{i} \E*{\abs*{\pderv*{X^{t,\vec x}_{i}}{x_{j}}\p{s}}^{p}} \leq 1 +
    \p*{c_{1} + 2 C_{1}^{p} c_{2}}
    \int_{t}^{s} \max_{i} \E*{ \abs*{\pderv*{X^{t,\vec x}_{i}}{x_{j}}\p{\tau} }^{p}} \D \tau \eqcomma
  \]
  and using \groneq{} yields
  \begin{equation}\label{eq:maxi-1st-var}
    \max_{i} \E*{\abs*{\pderv*{X^{t,\vec x}_{i}}{x_{j}}\p{s}}^{p}} \leq %
     \exp\p*{\p{c_{1} + 2 C_{1}^{p} c_{2} }\p{s-t}} \eqdef D_{1}\eqcomma
  \end{equation}
  for all \(t \le s \leq T\). Taking the maximum over all \(i\neq j\) in \cref{eq:final-1st-var}, we arrive at
  \[
    \begin{aligned}
      \max_{i\neq j} \E*{\abs*{\pderv*{X^{t,\vec x}_{i}}{x_{j}}\p{s}}^{p}} \leq&\
      \p{c_{1}+c_{2} C_{1}^{p}} \int_{t}^{s}
      \max_{i \ne j} \E*{ \abs*{\pderv*{X^{t,\vec x}_{i}}{x_{j}}\p{\tau} }^{p}} \D \tau \\
      &+ c_{2} \p*{\max_{i \neq j}{a_{i,j}^{p}}}  \int_{t}^{s}  \, \E*{ \abs*{\pderv*{X^{t,\vec x}_{j}}{x_{j}}\p{\tau} }^{p}} \D\tau \eqperiod
    \end{aligned}
  \]
  Using \groneq{} and \eqref{eq:maxi-1st-var} yields
  \begin{align}
    \notag\max_{i \ne j} \E*{\abs*{\pderv*{X^{t,\vec x}_{i}}{x_{j}}\p{s}}^{p}}
    &\leq c_{2} \p*{\max_{i \neq j}a_{i,j}^{p}} \exp\p*{\p{c_{1} + c_{2} C_{1}^{p}  }\p{s-t}}
      {\int_{t}^{s}  \E*{ \abs*{\pderv*{X^{t,\vec x}_{j}}{x_{j}}\p{\tau} }^{p}} \D\tau}\\
    \notag  &\leq %
              c_{2} \p*{\max_{i \neq j}a_{i,j}^{p}}
              \,\exp\p*{\p{c_{1} + c_{2}C_{1}^{p}  }\p{s-t}}
              \,  \p*{\int_{t}^{s}D_{1} \D\tau}
    \\
    &\eqdef \widetilde D_{1} \, \p*{\max_{i \neq j}a_{i,j}^{p}} \label{eq:maxij-1st-var} \eqperiod
  \end{align}
  Finally, using \eqref{eq:maxi-1st-var}  and \eqref{eq:maxij-1st-var} we arrive at:
  \begin{align}
    \notag
    \norm{\vec X^{t, \vec x}}_
    {\mathrm D^{1}, L^{\infty}\p{[t, T]; L^{p}\p{\Omega, \mathbb P}}} &= \max_{i} \sup_{t \leq s \leq T} \E*{\abs*{\pderv*{X^{t,\vec x}_{i}}{x_{i}}\p{s}}^{p}}^{1/p} + \sum_{j=1}^{\dim} \max_{i \neq j}  \sup_{t \leq s \leq T} \E*{\abs*{\pderv*{X^{t,\vec x}_{i}}{x_{j}}\p{s}}^{p}}^{1/p}\\
      &\notag\leq D_{1}^{1/p} +\widetilde D_{1}^{1/p} \sum_{j=1}^{\dim} \max_{i \neq j} a_{i,j} \\
      &\notag\leq D_{1}^{1/p} + \widetilde  D_{1}^{1/p} C_{1} \\
      &\eqdef K_{1,p}\eqperiod
        \label{eq:first-variation-bound}
  \end{align}
  \paragraph{Second variation}
  In this section, we simplify the presentation by using \(D_{2}\) to denote
  constants depending only on \(t, T, p,\) and \(C_{2}\) and independent of
  \(\dim\). Observe that these constants might change their values from one
  line to the next.
  Again, note that the process \(\br*{\pderv*{X^{t,\vec x}_{i}}{x_{j},
      x_{j'}}\p{s}}_{s \in [t,T]}\) exists under \cref{ass:sde-bounded-derv} for
  \(q=2\) and satisfies for \(s \in [t,T]\) the SDE
 \[
   \begin{aligned}
     \pderv*{X^{t,\vec x}_{i}}{x_{j},x_{j'}}\p{s} =  & \int_t^s \sum_{k=1}^{\dim} \pderv*{\gdrift_{i}}{x_{k}}\p{\tau,\vec X^{t,\vec x}\p{\tau}} \, \pderv*{X^{t,\vec x}_{k}}{x_{j},x_{j'}}\p{\tau} \, \D \tau \\
     &+ \int_t^s \sum_{k=1}^{\dim} \sum_{k'=1}^{\dim} \pderv*{\gdrift_{i}}{x_{k},x_{k'}}\p{\tau,\vec X^{t,\vec x}\p{\tau}} \, \pderv*{X^{t,\vec x}_{k}}{x_{j}}\p{\tau} \pderv*{X^{t,\vec x}_{k'}}{x_{j'}}\p{\tau} \, \D \tau \\
     &+ \sum_{m=1}^{\dim'} \int_t^s \sum_{k=1}^{\dim} \pderv*{\gdiffusion_{im}}{x_{k}} \p{\tau, \vec X^{t,\vec x}\p{\tau}} \pderv*{X^{t,\vec x}_{k}}{x_{j},x_{j'}}
     \p{\tau} \, \D W_{m}\p{\tau}\\
     &+ \sum_{m=1}^{\dim'}  \int_t^s \sum_{k=1}^{\dim}\sum_{k'=1}^{\dim} \pderv*{\gdiffusion_{im}}{x_{k}, x_{k'}} \p{\tau, \vec X^{t,\vec x}\p{\tau}} \pderv*{X^{t,\vec x}_{k}}{x_{j}}\p{\tau} \pderv*{X^{t,\vec x}_{k'}}{x_{j'}} \p{\tau} \, \D W_{m}\p{\tau}\eqcomma
   \end{aligned}
 \]
 cf. \cite[Theorem~5.4 in Chapter~5]{friedman:sde-book}. %
 By \ito's formula,
 \begin{equation}\label{eq:X-second-var-p-moment}
   \begin{aligned}
     \E*{\abs*{\pderv*{X^{t,\vec x}_{i}}{x_{j},x_{j'}}\p{s}}^{p}} \leq&&& %
     p \int_t^s \E*{\p{ f_{1} + f_{3}} \: \,
       \abs[\Big]{\pderv*{X^{t,\vec x}_{i}}{x_{j},x_{j'}}\p{\tau}}^{p-1}} \, \D \tau \\
     &&&+ \frac{p\p{p-1}}{2}\int_t^s \E*{\p{f_{2} + f_{4}}\: \,
       \abs[\Big]{\pderv*{X^{t,\vec x}_{i}}{x_{j},x_{j'}}\p{\tau}}^{p-2}
     } \, \D \tau\eqcomma %
   \end{aligned}
 \end{equation}
 where
 \[
   \begin{aligned}
     f_{1} &\defeq \abs*{\sum_{k=1}^{\dim} \pderv*{\gdrift_{i}}{x_{k}}\p{\tau,\vec X^{t,\vec x}\p{\tau}} \, \pderv*{X^{t,\vec x}_{k}}{x_{j},x_{j'}}\p{\tau}} \eqcomma \\
     f_{2} &\defeq \sum_{m=1}^{\dim'} \p*{ \sum_{k=1}^{\dim} \p*{ \pderv*{\gdiffusion_{im}}{x_{k}} \p{\tau, \vec X^{t,\vec x}\p{\tau}}} \pderv*{X^{t,\vec x}_{k}}{x_{j},x_{j'}}
       \p{\tau}}^{2}\eqcomma  \\
     f_{3} &\defeq \abs*{\sum_{k=1}^{\dim} \sum_{k'=1}^{\dim} \pderv*{\gdrift_{i}}{x_{k},x_{k'}}\p{\tau,\vec X^{t,\vec x}\p{\tau}} \, \pderv*{X^{t,\vec x}_{k}}{x_{j}}\p{\tau} \pderv*{X^{t,\vec x}_{k'}}{x_{j'}}\p{\tau}}\\
     \eqand f_{4} &\defeq \sum_{m=1}^{\dim'} \p*{\sum_{k=1}^{\dim}\sum_{k'=1}^{\dim} \p*{ \pderv*{\gdiffusion_{im}}{x_{k}, x_{k'}} \p{\tau, \vec X^{t,\vec x}\p{\tau}}} \pderv*{X^{t,\vec x}_{k}}{x_{j}}\p{\tau} \pderv*{X^{t,\vec x}_{k'}}{x_{j'}} \p{\tau} }^{2}\eqperiod
   \end{aligned}
 \]
 As before, the first step is to apply Young's inequality to each of the previous integrands.
 For integers \(q_{u} \in \br{1,2}\) and \(u \in \br{1,2,3,4}\), we have
 \begin{equation}\label{eq:f-youngs-ineq}
   \E*{f_{u} \: \abs*{\pderv*{X^{t,\vec x}_{i}}{x_{j},x_{j'}}\p{\tau}}^{p-q_{u}}} \leq
   \frac{q_{u}}{p}\E*{f_{u}^{p/q_{u}}} + \frac{p-q_{u}}{p}\E*{\abs*{\pderv*{X^{t,\vec x}_{i}}{x_{j},x_{j'}}\p{\tau}}^{p}}\eqperiod
 \end{equation}
 We now turn our attention to bounding \(\E{f_{u}^{p/q_{u}}}\) by primarily
 using Jensen's inequality and \cref{eq:jensen-ineq}.
 We present the proof bounding \(\E{f_{2}^{p/2}}\) and \(\E{f_{4}^{p/2}}\),
 Bounding \(\E{f_{1}^{p}}\) and \(\E{f_{3}^{p}}\) is analogous,
 \begin{align}
   \notag \E{f_{2}^{p/2}} &\leq  \E*{\p*{\sum_{m=1}^{\dim'} \p*{\sum_{k=1}^{\dim} \norm*{\pderv*{\gdiffusion_{im}}{x_{k}}}_{\bddspace} \,
                            \abs*{\pderv*{X^{t,\vec x}_{k}}{x_{j},x_{j'}}\p{\tau}}}^{2}}^{p/2}}\\
   \notag &\leq 3^{p-1} \sum_{k \in \br{j, j'}} \p*{\sum_{m=1}^{\dim'} \norm*{\pderv*{\gdiffusion_{im}}{x_{k}}}_{\bddspace}^{2}}^{p/2} \, \E*{ \abs*{\pderv*{X^{t,\vec x}_{k}}{x_{j},x_{j'}}\p{\tau}}^{p}}
   \\
   \notag            &\quad + 3^{p-1} \E*{\p*{\sum_{m=1}^{\dim'}\p*{\sum_{k \in \br{1,2,\ldots, \dim} \setminus \br{j, j'}} \norm*{\pderv*{\gdiffusion_{im}}{x_{k}}}_{\bddspace} \, \abs*{\pderv*{X^{t,\vec x}_{k}}{x_{j},x_{j'}}\p{\tau}}}^{2}}^{p/2}}
   \\
   \notag &\leq 3^{p-1} \sum_{k \in \br{j, j'}} \p*{\sum_{m=1}^{\dim'} \norm*{\pderv*{\gdiffusion_{im}}{x_{k}}}_{\bddspace}^{2}}^{p/2} \, \p*{\E*{ \abs*{\pderv*{X^{t,\vec x}_{k}}{x_{j},x_{j'}}\p{\tau}}^{p}}}
   \\
                          &\quad + 3^{p-1} C_{1}^{p} \max_{k \in \br{1,2,\ldots, \dim} \setminus \br{j, j'}} \E*{\abs*{\pderv*{X^{t,\vec x}_{k}}{x_{j},x_{j'}}\p{\tau}}^{p}}\eqcomma
                                 \label{eq:f2-bound}
 \end{align}
 where we used \cref{eq:jensen-ineq} in the last step. On the other hand,
 \begin{align}
     \notag \E{f_{4}^{p/2}} &\leq \E*{
       \p*{\sum_{m=1}^{\dim'} \p*{\sum_{k=1}^{\dim} \sum_{k'=1}^{\dim} \norm*{\pderv*{\gdiffusion_{im}}{x_{k},x_{k'}}}_{\bddspace} \, \pderv*{X^{t,\vec x}_{k}}{x_{j}}\p{\tau} \pderv*{X^{t,\vec x}_{k'}}{x_{j'}}\p{\tau}}^{2}}^{p/2}}\\
     \notag &\leq 4^{p-1}
     \p*{\sum_{m=1}^{\dim'} \norm*{\pderv*{\gdiffusion_{im}}{x_{j},x_{j'}}}^{2}_{\bddspace}}^{p/2}\: \E*{\abs*{\pderv*{X^{t,\vec x}_{j}}{x_{j}}\p{\tau} \pderv*{X^{t,\vec x}_{j'}}{x_{j'}}\p{\tau}}^{p}}\\
     \notag &\quad+4^{p-1} \E*{ \p*{\sum_{m=1}^{\dim'}
         \p*{\sum_{k'=1, k'\neq j'}^{\dim} \norm*{\pderv*{\gdiffusion_{im}}{x_{j},x_{k'}}}_{\bddspace} \,  \pderv*{X^{t,\vec x}_{k'}}{x_{j'}}\p{\tau}}^{2}}^{p/2} \abs*{\pderv*{X^{t,\vec x}_{j}}{x_{j}}\p{\tau}}^{p}}\\
     \notag &\quad+4^{p-1} \E*{ \p*{\sum_{m=1}^{\dim'}
         \p*{\sum_{k=1, k\neq j}^{\dim} \norm*{\pderv*{\gdiffusion_{im}}{x_{k},x_{j'}}}_{\bddspace} \,  \pderv*{X^{t,\vec x}_{j}}{x_{k}}\p{\tau}}^{2}}^{p/2} \abs*{\pderv*{X^{t,\vec x}_{j'}}{x_{j'}}\p{\tau}}^{p}}\\
                             &\quad+4^{p-1} \E*{ \p*{\sum_{m=1}^{\dim'} \p*{\sum_{k=1, k\neq j}^{\dim} \sum_{k'=1, k'\neq j'}^{\dim} \norm*{\pderv*{\gdiffusion_{im}}{x_{k},x_{k'}}}_{\bddspace} \, \pderv*{X^{t,\vec x}_{k}}{x_{j}}\p{\tau} \pderv*{X^{t,\vec x}_{k'}}{x_{j'}}\p{\tau}}^{2}}^{p/2}}\eqperiod
                                  \label{eq:f4-bound}
 \end{align}
 Looking at each term separately and using the bound on the first variation
 \cref{eq:first-variation-bound},
 \[
   \begin{aligned}
     \E*{\abs*{\pderv*{X^{t,\vec x}_{j}}{x_{j}}\p{\tau} \pderv*{X^{t,\vec x}_{j'}}{x_{j'}}\p{\tau}}^{p}} &\leq
     \p*{\E*{\abs*{\pderv*{X^{t,\vec x}_{j}}{x_{j}}\p{\tau}}^{2p}}}^{1/2}
     \p*{\E*{\abs*{\pderv*{X^{t,\vec x}_{j'}}{x_{j'}}\p{\tau}}^{2p}}}^{1/2}\\
     & \leq K_{1, 2p}^{2p}\eqperiod
   \end{aligned}
 \]
 Moreover, using \holdereq{}, \cref{eq:jensen-ineq} and the bound on the first
 variation \cref{eq:first-variation-bound},
 \[
   \begin{aligned}
     &\E*{ \p*{\sum_{m=1}^{\dim'}
         \p*{\sum_{k'=1, k'\neq j'}^{\dim} \norm*{\pderv*{\gdiffusion_{im}}{x_{j},x_{k'}}}_{\bddspace} \,  \pderv*{X^{t,\vec x}_{k'}}{x_{j'}}\p{\tau}}^{2}}^{p/2} \abs*{\pderv*{X^{t,\vec x}_{j}}{x_{j}}\p{\tau}}^{p}}\\
     \begin{details}
       &\leq\E*{\p*{\sum_{m=1}^{\dim'}
           \p*{\sum_{k'=1, k'\neq j'}^{\dim} \norm*{\pderv*{\gdiffusion_{im}}{x_{j},x_{k'}}}_{\bddspace} \,  \pderv*{X^{t,\vec x}_{k'}}{x_{j'}}\p{\tau}}^{2}}^{p}}^{1/2} \E*{\abs*{\pderv*{X^{t,\vec x}_{j}}{x_{j}}\p{\tau}}^{2p}}^{1/2}\\
     \end{details}
     &\leq \p*{\sum_{k'=1}^{\dim} \p*{\sum_{m=1}^{\dim'} \norm*{\pderv*{\gdiffusion_{im}}{x_{j},x_{k'}}}_{\bddspace}^{2}}^{1/2}}^{p}
     \p*{ \max_{k'\neq j'} \E*{\abs*{\pderv*{X^{t,\vec x}_{k'}}{x_{j'}}\p{\tau}}^{2p}} }^{1/2}\:
     \E*{\abs*{\pderv*{X^{t,\vec x}_{j}}{x_{j}}\p{\tau}}^{2p}}^{1/2}\\
     &\leq K_{1,2p}^{p}
     \p*{\sum_{k'=1}^{\dim} \p*{\sum_{m=1}^{\dim'} \norm*{\pderv*{\gdiffusion_{im}}{x_{j},x_{k'}}}_{\bddspace}^{2}}^{1/2}}^{p}
     \p*{ \max_{k'\neq j'} \E*{\abs*{\pderv*{X^{t,\vec x}_{k'}}{x_{j'}}\p{\tau}}^{2p}} }^{1/2}\eqcomma
   \end{aligned}
 \]
 and
 \[
   \begin{aligned}
     &\E*{
       \p*{\sum_{m=1}^{\dim'} \p*{\sum_{k=1, k\neq j}^{\dim} \sum_{k'=1, k'\neq j'}^{\dim} \norm*{\pderv*{\gdiffusion_{im}}{x_{k},x_{k'}}}_{\bddspace} \, \pderv*{X^{t,\vec x}_{k}}{x_{j}}\p{\tau} \pderv*{X^{t,\vec x}_{k'}}{x_{j'}}\p{\tau}}^{2}}^{p/2}}\\
     &\leq \p*{\sum_{k=1, k\neq j}^{\dim} \sum_{k'=1, k'\neq j'}^{\dim}
       \p*{\sum_{m=1}^{\dim'} \norm*{\pderv*{\gdiffusion_{im}}{x_{k},x_{k'}}}_{\bddspace}^{2}}^{1/2} }^{p}
     \max_{\substack{k, k' \\
         k\neq j, k'\neq j'}}\E*{\abs*{\, \pderv*{X^{t,\vec x}_{k}}{x_{j}}\p{\tau} \pderv*{X^{t,\vec x}_{k'}}{x_{j'}}\p{\tau}}^{p}} \\
     &\leq C_{2}^{p} \p*{\max_{\substack{k, k\neq j}} \E*{\abs*{\, \pderv*{X^{t,\vec x}_{k}}{x_{j}}\p{\tau}}^{2p}}}^{1/2} \p*{\max_{\substack{k', k'\neq j'}} \E*{\abs*{ \pderv*{X^{t,\vec x}_{k'}}{x_{j'}}\p{\tau}}^{2p}}}^{1/2}\eqperiod
   \end{aligned}
 \]
 Hence
  \[
   \begin{aligned}
     \E{f_{4}^{p/2}} &\leq 4^{p-1} \p* {K_{1,2p}^{2p} \, F_{\gdiffusion_{1}, i,j,j'} + K_{1,2p}^{p}\, F_{\gdiffusion_{2}, i,j,j'} + K_{1,2p}^{p}\, F_{\gdiffusion_{2}, i,j',j} + C_{2}^{p} \, F_{3, j,j'}} \\
     &\leq D_{2} \p* {F_{\gdiffusion_{1}, i,j,j'} + F_{\gdiffusion_{2}, i,j,j'} + F_{\gdiffusion_{2}, i,j',j} + F_{3,j,j'}}\eqcomma
   \end{aligned}
 \]
 where
 \[
   \begin{aligned}
     F_{\gdiffusion_{1}, i,j,j'} &\defeq \p*{\sum_{m=1}^{\dim'} \norm*{\pderv*{\gdiffusion_{im}}{x_{j},x_{j'}}}_{\bddspace}^{2}}^{p/2}\eqcomma\\
     F_{\gdiffusion_{2}, i,j,j'} &\defeq \p*{\sum_{k'=1}^{\dim} \p*{\sum_{m=1}^{\dim'} \norm*{\pderv*{\gdiffusion_{im}}{x_{j},x_{k'}}}_{\bddspace}^{2}}^{1/2}}^{p}
     \p*{ \max_{k'\neq j'} \E*{\abs*{\pderv*{X^{t,\vec x}_{k'}}{x_{j'}}\p{\tau}}^{2p}} }^{1/2}\eqcomma\\
     \eqand F_{3,j,j'} &\defeq
     \p*{\max_{\substack{k\neq j}} \E*{\abs*{\, \pderv*{X^{t,\vec x}_{k}}{x_{j}}\p{\tau}}^{2p}}}^{1/2} \p*{\max_{\substack{k'\neq j'}} \E*{\abs*{ \pderv*{X^{t,\vec x}_{k'}}{x_{j'}}\p{\tau}}^{2p}}}^{1/2}\eqperiod
   \end{aligned}
 \]
 Note that, using \cref{ass:sde-bounded-derv},
 \[
   \begin{aligned}
     &\sum_{j=1}^{\dim} \max_{i\neq j} F_{\gdiffusion_{1}, i, j, i} + \sum_{j=1}^{\dim} \sum_{j'=1}^{\dim} \max_{i \notin \br{j, j'}} F_{\gdiffusion_{1}, i, j, j'}\\
     &\leq \sum_{j=1}^{\dim}  \max_{i\neq j} \p*{\sum_{m=1}^{\dim'} \norm*{\pderv*{\gdiffusion_{im}}{x_{i},x_{j}}}_{\bddspace}^{2}}^{p/2}
     + \sum_{j=1}^{\dim} \sum_{j'=1}^{\dim} \max_{i \notin \br{j, j'}} \p*{\sum_{m=1}^{\dim'} \norm*{\pderv*{\gdiffusion_{im}}{x_{j},x_{j'}}}_{\bddspace}^{2}}^{p/2}\\
     &\leq C_{2}^{p}\eqperiod
   \end{aligned}
\]
Similarly, using \cref{ass:sde-bounded-derv}, \cref{eq:max-seq} and
\cref{eq:first-variation-bound},
\[
  \begin{aligned}
    & \sum_{j=1}^{\dim} \max_{i \neq j} F_{\gdiffusion_{2}, i,j,i} +\sum_{j=1}^{\dim} \sum_{j'=1}^{\dim} \max_{i \notin \br{j,j'}} F_{\gdiffusion_{2}, i,j,j'}
    \\
    &= \p*{ \sum_{j=1}^{\dim} \max_{i \neq j} \sum_{k'=1}^{\dim} \p*{\sum_{m=1}^{\dim'} \norm*{\pderv*{\gdiffusion_{im}}{x_{j},x_{k'}}}_{\bddspace}^{2}}^{1/2}}^{p}
      \times \p*{\max_{i, k'} \p*{\E*{\abs*{\pderv*{X^{t,\vec x}_{k'}}{x_{i}}\p{\tau}}^{2p}}}^{1/2} } \\
    &\quad +
      \p*{ \sum_{j=1}^{\dim} \max_{i \neq j}\sum_{k'=1}^{\dim} \p*{\sum_{m=1}^{\dim'} \norm*{\pderv*{\gdiffusion_{im}}{x_{j},x_{k'}}}_{\bddspace}^{2}}^{1/2}}^{p} %
      \times \p*{\sum_{j'=1}^{\dim} \max_{k'\neq j'} \p*{\E*{\abs*{\pderv*{X^{t,\vec x}_{k'}}{x_{j'}}\p{\tau}}^{2p}}}^{1/2} }\\
    &\leq 2 C_{2}^{p} \, K_{1,2p}^{p}\eqperiod
  \end{aligned}
\]
Additionally, using \cref{eq:first-variation-bound},
\[
  \begin{aligned}
    \sum_{j=1}^{\dim} \sum_{j'=1}^{\dim} F_{3, j,j'}
    &\leq \p*{\sum_{j=1}^{\dim} \max_{\substack{k\neq j}} \p*{\E*{\abs*{\, \pderv*{X^{t,\vec x}_{k}}{x_{j}}\p{\tau}}^{2p}}}^{1/2}}^{p}
      \p*{ \sum_{j'=1}^{\dim} \max_{\substack{k'\neq j'}} \p*{\E*{\abs*{ \pderv*{X^{t,\vec x}_{k'}}{x_{j'}}\p{\tau}}^{2p}}}^{1/2}}\\
    &\leq K_{1,2p}^{2p}\eqperiod
  \end{aligned}
\]
Similarly, we bound \(\E{f_{3}^{p}}\) to arrive at
 \[
   \E{f_{3}^{p}} \le D_{2} \p* { F_{\gdrift_{1}, i,j,j'} + F_{\gdrift_{2}, i,j,j'} + F_{\gdrift_{2}, i,j',j} + F_{3, j,j'}}\eqcomma
 \]
 where
  \[
   \begin{aligned}
     F_{\gdrift_{1}, i,j,j'} &\defeq \norm*{\pderv*{\gdrift_{i}}{x_{j},x_{j'}}}_{\bddspace}^{p}\\
     \eqand F_{\gdrift_{2}, i,j,j'} &\defeq \p*{\sum_{k'=1}^{\dim} \norm*{\pderv*{\gdrift_{i}}{x_{j},x_{k'}}}_{\bddspace}}^{p}
     \p*{ \max_{k'\neq j'} \E*{\abs*{\pderv*{X^{t,\vec x}_{k'}}{x_{j'}}\p{\tau}}^{2p}} }^{1/2}\eqperiod
   \end{aligned}
 \]
 Therefore we can find \(b_{i,j,j'}\) such that
\begin{equation}\label{eq:f3-f4-bound}
  \int_{t}^{s} \E{f_{3}^{p}} + \p{p-1} \E{f_{4}^{p/2}} \D \tau
    \leq b_{i,j,j'}^{p}  \eqcomma
\end{equation}
which satisfy
\begin{equation}\label{eq:2nd-var-b-bound}
   \sum_{j=1}^{\dim} \p*{\max_{i \neq j} b_{i,j,i}}
   + \sum_{j=1}^{\dim} \sum_{j'=1}^{\dim}  \p*{\max_{i \notin \br{j, j'}} b_{i,j,j'}}
   \leq  D_{2}\eqperiod
 \end{equation}
 We use \cref{eq:f2-bound} and \cref{eq:f3-f4-bound}
 in~\cref{eq:f-youngs-ineq} and the result in \cref{eq:X-second-var-p-moment}
 and simplify to arrive at
\begin{equation}\label{eq:2nd-var-gen-bound}
  \begin{aligned}
    \E*{\abs*{\pderv*{X^{t,\vec x}_{i}}{x_{j},x_{j'}}\p{s}}^{p}} &\leq
    D_{2} \int_{t}^{s} \E*{\abs*{\pderv*{X^{t,\vec x}_{i}}{x_{j},x_{j'}}\p{\tau}}^{p}}\D \tau \\
    &\quad+ D_{2} \int_{t}^{s} \max_{k \in \br{1,2,\ldots, \dim} \setminus \br{j, j'}}
    \E*{\abs*{\pderv*{X^{t,\vec x}_{k}}{x_{j},x_{j'}}\p{\tau}}^{p}} \D\tau\\
    &\quad+D_{2}\, a_{i,j}^{p} \int_{t}^{s} %
    \E*{ \abs*{\pderv*{X^{t,\vec x}_{j}}{x_{j},x_{j'}}\p{\tau}}^{p}} \D\tau\\
    &\quad+D_{2} a_{i,j'}^{p} \int_{t}^{s} %
    \E*{ \abs*{\pderv*{X^{t,\vec x}_{j'}}{x_{j},x_{j'}}\p{\tau}}^{p}} \D\tau\\
    &\quad+ b_{i,j,j'}\eqperiod
  \end{aligned}
\end{equation}
Recall that
 \[
   \begin{aligned}
       a_{i,j}^{p}&\defeq
     \norm*{\pderv*{\gdrift_{i}}{x_{j}}}_{\bddspace}^{p} +
     \p*{\sum_{m=1}^{\dim'} \norm*{\pderv*{\gdiffusion_{im}}{x_{j}}}_{\bddspace}^{2}}^{p/2} \eqcomma
   \end{aligned}
 \]
 and \(\max_{i} a_{i,i} + \sum_{j=1}^{\dim} \max_{i \ne j} a_{i,j} \leq 2 C_{1}\).
 Then, taking the maximum over all \(i\) in \cref{eq:2nd-var-gen-bound} and
 using \groneq{} yield
 \begin{equation}\label{eq:2nd-var-max-bound}
   \max_{i}
   \E*{\p*{\pderv*{X^{t,\vec x}_{i}}{x_{j},x_{j'}}\p{s}}^{p}} \leq
   D_{2} %
   \eqcomma
 \end{equation}
 for all \(j, j' \in \br{1, \ldots, \dim}\). Setting \(j'=i\) and taking the maximum over \(i \neq j\) in \cref{eq:2nd-var-gen-bound} yield
 \[
   \begin{aligned}
     \max_{i \ne j}\E*{\abs*{\pderv*{X^{t,\vec x}_{i}}{x_{j},x_{i}}\p{s}}^{p}} \leq&&&
     D_{2}  \int_{t}^{s} \max_{i \ne j} \E*{\abs[\Big]{\pderv*{X^{t,\vec x}_{i}}{x_{j},x_{i}}\p{\tau}}^{p}}\D \tau \\
     &&+&D_{2} \,\p*{\max_{i \ne j} a_{i,j}^{p}} \int_{t}^{s}
     \max_{i} \E*{ \abs*{\pderv*{X^{t,\vec x}_{j'}}{x_{j'},x_{i}}\p{\tau}}^{p}} \D\tau\\
     &&+&D_{2}  \int_{t}^{s}
     \max_{i \ne j} \E*{ \abs*{\pderv*{X^{t,\vec x}_{i}}{x_{j},x_{i}}\p{\tau}}^{p}} \D\tau\\
     &&+& \max_{i \ne j} b_{i,j,i}^{p}\eqperiod
   \end{aligned}
 \]
 Using \eqref{eq:2nd-var-max-bound} to bound the second term, then \groneq{}
 and then taking the \(\nth{p}\) root and summing over \(j\) yields
 \begin{equation}\label{eq:2nd-var-single-sum-bound}
   \begin{aligned}
     \sum_{j=1}^{\dim} \max_{i \neq j} \E*{\abs*{\pderv*{X^{t,\vec x}_{i}}{x_{i},x_{j}}\p{s}}^{p}}^{1/p}
     \leq&&& %
     D_{2} \p*{\sum_{j=1}^{\dim} \p*{\max_{i \neq j} a_{i,j} } +
          \sum_{j=1}^{\dim} \p*{\max_{i \neq j} b_{i,j,i}}}
          \leq D_{2}  \eqcomma
   \end{aligned}
 \end{equation}
where we used \eqref{eq:2nd-var-b-bound}.
 Finally, taking the maximum over \(i \notin \br{j,j'}\) in \cref{eq:2nd-var-gen-bound},
 \[
     \begin{aligned}
    \max_{i\notin \br{j,j'}} \E*{\abs*{\pderv*{X^{t,\vec x}_{i}}{x_{j},x_{j'}}\p{s}}^{p}} \leq&&&
    D_{2} \, \int_{t}^{s} \max_{i\notin \br{j,j'}} \E*{\abs[\Big]{\pderv*{X^{t,\vec x}_{i}}{x_{j},x_{j'}}\p{\tau}}^{p}}\D \tau \\
    &&+&D_{2}\, \p*{\max_{i \neq j} a_{i,j}^{p}}  \int_{t}^{s}  %
    \E*{ \abs*{\pderv*{X^{t,\vec x}_{j}}{x_{j},x_{j'}}\p{\tau}}^{p}} \D\tau\\
    &&+&D_{2}\,\p*{\max_{i \neq j'} a_{i,j'}^{p}}  \int_{t}^{s} %
    \E*{ \abs*{\pderv*{X^{t,\vec x}_{j'}}{x_{j},x_{j'}}\p{\tau}}^{p}} \D\tau\\
    &&+& \max_{i\notin \br{j,j'}} b_{i,j,j'}^{p}\eqperiod
  \end{aligned}
 \]
 Then, using \groneq{}, taking the \(\nth{p}\) root and summing over \(j\) and \(j'\) yield
 \[
   \begin{aligned}
     &&&\sum_{j=1}^{\dim} \sum_{j'=1}^{\dim} \max_{i \notin \br{j, j'}} \E*{\abs*{\pderv*{X^{t,\vec x}_{i}}{x_{j},x_{j'}}\p{s}}^{p}}^{1/p}\\
     \leq &&& D_{2} \p*{\sum_{j=1}^{\dim} \max_{i \ne j} a_{i,j}} \p*{ \max_{j} \sum_{j'=1}^{\dim} {
         \sup_{t \leq \tau \leq s}\E*{ \abs*{\pderv*{X^{t,\vec x}_{j}}{x_{j},x_{j'}}\p{\tau}}^{p}}}^{1/p}} \\
     &&+& D_{2} \p*{\sum_{j'=1}^{\dim} \max_{i \ne j'} a_{i,j'}} \p*{ \max_{j'} \sum_{j=1}^{\dim}
       \sup_{t \leq \tau \leq s}\E*{ \abs*{\pderv*{X^{t,\vec x}_{j'}}{x_{j},x_{j'}}\p{\tau}}^{p}}^{1/p}} \\
    &&+& D_{2} \sum_{j=1}^{\dim} \sum_{j'=1}^{\dim}  \p*{\max_{i \notin \br{j, j'}} b_{i,j,j'}} \eqperiod
   \end{aligned}
 \]
 The result follows by \cref{eq:2nd-var-b-bound}, and since \(\sum_{j=1}^{\dim}
 \max_{i \ne j} a_{i,j} \le C_{1}\) and, by \cref{eq:max-seq},
 \cref{eq:2nd-var-max-bound} and \cref{eq:2nd-var-single-sum-bound},
 \[
   \begin{aligned}
     &\max_{j} \sum_{j'=1}^{\dim}  \sup_{t \leq \tau \leq s}\E*{ \abs*{\pderv*{X^{t,\vec x}_{j}}{x_{j},x_{j'}}\p{\tau}}^{p}} \\
     \leq &\max_{j}  \sup_{t \leq \tau \leq s}\E*{ \abs*{\pderv*{X^{t,\vec x}_{j}}{x_{j},x_{j}}\p{\tau}}^{p}}+ \sum_{j'=1}^{\dim}   \max_{j \neq j'}%
     \sup_{t \leq \tau \leq s}  \E*{ \abs*{\pderv*{X^{t,\vec x}_{j}}{x_{j},x_{j'}}\p{\tau}}^{p}}\\
     \leq & D_{2}\eqperiod
  \end{aligned}
 \]
  \paragraph{Third variation}
 The result for the third variations can be proven similarly and is omitted here.%
\end{proof}
\begin{proposition}[Bounds on derivatives of the value function]\label{lemma:FK-Dell-bound}
  Let \(u: [0,T] \times \rset^{\dim} \to \rset\) satisfy the Kolmogorov backward
  equation on \(\p{t, \vec x} \in [0, T) \times \rset^{\dim}\),
  \begin{equation}\label{eq:general-feynman-kac}
    \begin{aligned}
      \frac{\partial u}{\partial t}\p{t, \vec x} +
        \sum_{i=1}^\dim \gdrift_i\p{t, \vec x} \frac{\partial u}{\partial x_i}\p{t, \vec x} +
    \frac{1}{2} \sum_{i=1}^\dim \sum_{j=1}^\dim
    \p*{\sum_{m=1}^{\dim'} \gdiffusion_{im}\p{t, \vec x} \gdiffusion_{jm}\p{t, \vec x}}
    \pderv*{u}{x_i, x_j}\p{t, \vec x} &= 0\\
      \eqand u\p{T, \vec x} &= g\p{\vec x}\eqperiod
    \end{aligned}
  \end{equation}
  Assume that the coefficients, \(\br{\gdrift_{i}}_{i=1}^{\dim}\) and
  \(\br{\gdiffusion_{i,m}}_{i \in {1, \ldots, \dim}, m \in {1, \ldots, \dim }}\), satisfy
  \cref{ass:sde-bounded-derv} for \(q \in \br{2,3}\) and that \(g\) has
  continuous bounded derivatives up to order \(q\). Then, for some constant
  \(D_{q}\), independent of \(\dim\), there holds
  \[
    \seminorm{u\p{t, \cdot}}_{\Cbspace^{q}\p{\rset^{\dim}}} \leq D_{q}\: \seminorm{g}_{\Cbspace^{q}\p{\rset^{\dim}}}
    \eqperiod
  \]
\end{proposition}
\begin{proof}
  First note that under~\cref{ass:sde-bounded-derv} for \(q=2\), \(u\)
  satisfies for all \(t \in [0,T]\) and \(\vec x \in \rset^\dim\)
  \cite[Theorem~6.1 in Chapter~5]{friedman:sde-book}
\begin{equation}\label{eq:FK-exp-bound}
  u\p{t, \vec x} = \E*{g\p{\vec X^{t,\vec x}\p{T}}}\eqperiod
\end{equation}
Next, we differentiate \(u\) with respect to the initial conditions and
exchange the differentiation with the expectation in~\cref{eq:FK-exp-bound}
\cite[Theorem~5.5 in Chapter~5]{friedman:sde-book}.
Then, we bound
\[
  \begin{aligned}
    \sum_{j=1}^{\dim} \abs*{\pderv*{u}{x_{j}}\p{t,\vec x}} &= \sum_{j=1}^{\dim}
    \abs*{\E*{\sum_{i=1}^{\dim}\pderv*{g}{x_{i}}\p{\vec X^{t,\vec x}\p{T}}
        \: \pderv*{X^{t,\vec x}_{i}}{x_{j}}\p{T}}}\\
    &\leq \sum_{j=1} ^{\dim} \sum_{i=1}^{\dim}
    \norm*{\pderv*{g}{x_{i}}}_{C\p{\rset^{\dim}}} \:
    \E*{\abs*{\pderv*{X^{t,\vec x}_{i}}{x_{j}}\p{T}}}\\
    &\leq \sum_{i=1}^{\dim} \norm*{\pderv*{g}{x_{i}}}_{C\p{\rset^{\dim}}} \:
    \p*{\sum_{j=1}^{\dim}
      \E*{\abs*{\pderv*{X^{t,\vec x}_{i}}{ x_{j}}\p{T}}}}\\
    &\leq \p*{\sum_{i=1}^{\dim} \norm*{\pderv*{g}{x_{i}}}_{C\p{\rset^{\dim}}}}
    \p*{\max_{i} \sum_{j=1}^{\dim} \E*{\abs*{\pderv*{X^{t,\vec x}_{i}}{x_{j}}\p{T}}}} \\
    &\leq K_{1,2} \: \p*{\sum_{i=1}^{\dim}
      \norm*{\pderv*{g}{x_{i}}}_{C\p{\rset^{\dim}}}}\eqcomma
  \end{aligned}
\]
by \cref{prop:sde-variation}. Similarly,
\[
  \begin{aligned}
    \sum_{j=1}^{\dim} \sum_{j'=1}^{\dim} \abs*{\pderv*{u}{x_{j}, x_{j'}}\p{t,\vec x}}
    &\leq
      \sum_{j=1}^{\dim}\sum_{j'=1}^{\dim} \abs*{\E*{ \sum_{i=1}^{\dim}\pderv*{g}{x_{i}}\p{\vec
      X^{t,\vec x}\p{T}} \: \pderv*{X^{t,\vec x}_{i}}{x_{j}, x_{j'}}\p{T} }}\\%
    &+ \sum_{j=1}^{\dim}\sum_{j'=1}^{\dim} \abs*{\E*{ \sum_{i=1}^{\dim} \sum_{i'=1}^{\dim}\pderv*{g}{x_{i}, x_{i'}}\p{\vec X^{t,\vec x}\p{T}} \: \pderv*{X_{i}^{t,\vec x}}{x_{j}}\p{T}
      \pderv*{X_{i'}^{t,\vec x}}{x_{j}}\p{T} }}\\
    &\leq \p*{\sum_{i=1}^{\dim}\norm*{\pderv*{g}{x_{i}}}_{C\p{\rset^{\dim}}}}
      \p*{\max_{i} \sum_{j=1}^{\dim}\sum_{j'=1}^{\dim} \E*{ \abs*{\pderv*{X^{t,\vec x}_{i}}{x_{j}, x_{j'}}\p{T}} }}\\%
    &+ \p*{\sum_{i=1}^{\dim} \sum_{i'=1}^{\dim}\norm*{\pderv*{g}{x_{i},
      x_{i'}}}_{C\p{\rset^{\dim}}}} \p*{\max_{i} \sum_{j=1}^{\dim} \E*{
      \p*{\pderv*{X^{t,\vec x}_{i}}{x_{j}}\p{T}}^{2}}^{1/2}}^{2}
    \\
    &\leq \p{K_{1,2}^{2}+K_{2,2}}
    \p*{\sum_{i=1}^{\dim}\norm*{\pderv*{g}{x_{i}}}_{C\p{\rset^{\dim}}} +
      \sum_{i=1}^{\dim} \sum_{i'=1}^{\dim}\norm*{\pderv*{g}{x_{i},
          x_{i'}}}_{C\p{\rset^{\dim}}} }\eqperiod
  \end{aligned}
\]
It is easy to see that the previous proof extends to \(q=3\) as well.%

\end{proof}
 
\section{Numerical Verification}\label{sec:numerics} %
In this section, we present a numerical study of the weak error of a particle
approximation of the solution of a simple McKean-Vlasov
equation. %
For \(r \in \nset\) and \(x \in \rset\), consider the function
\[
  \psi_{r}\p{x} \defeq
  \begin{cases}
    \p*{1 - x^{2}}^{r} & \abs{x} \leq 1\eqcomma\\
    0 & \abs{x} > 1\eqperiod
  \end{cases}
\]
and note that \(\psi_{0}\) is discontinuous while for \(r > 0\) and all \(k < r\)
the \(\nth{k}\) derivative of \(\psi_{r}\) exists is uniformly bounded, and the
\(\nth{\p{r-1}}\) derivative is Lipschitz continuous. Subsequently, consider
the
particle system~\cref{eq:X-sys} with \(d=1\) and \(X_{i}(0)\) being uniformly
distributed in \([-1,1]\), and for \(x,y \in \rset\) set
\begin{equation}\label{eq:example-particle-system}
  \begin{aligned}
    \drift\p{x,y} &\defeq 2 \p{x-0.2} + y \eqcomma\\
    \diffusion\p{x,y} &\defeq 0.2\p{1+y}\eqcomma\\
    \kernel_{1}\p{x,y} &\defeq %
    \psi_{1}\p{10 \abs{x-y}}
    \\
    \eqand\kernel_{2}\p{x,y} &\defeq
    \psi_{1}\p{5 \abs{x-y}}
    \eqperiod
  \end{aligned}
\end{equation}
Note that the previous \(\drift, \diffusion, \kernel_{1},\) and
\(\kernel_{2}\), the latter two being only Lipschitz continuous, do not satisfy the
conditions the conditions of \cref{thm:weak-conv}. To approximate solutions
to~\cref{eq:X-sys}, we use an Euler--Maruyama time-stepping scheme with a fixed
number of time-steps.

To illustrate the convergence of \(\vec X^{\P}\) to \(Z\), the corresponding
solution of the McKean--Vlasov equation \cref{eq:Z-MV}, we consider the
sequence of systems, denoted by \(\vec X^{\P}\), satisfying \cref{eq:X-sys}
with an increasing number of particles, \(\P\). See \cref{fig:histogram} for a
histogram of the values of \(X^{\P}_{i}\p{1}\) for \(\P=2,048\) and using
\(64\) uniform time-steps in an Euler-Maruyama scheme. We also consider the
discontinuous function \( g\p{x} = %
\psi_{0}\p{10 \abs{x-0.2}}\) and plot in \cref{fig:conv-rates} the quantities
\(\p*{\E{ \p{X^{2 \P}_{i} - X^{\P}_{i}}^{2} }}^{1/2}\) and
\(\abs{\E{g\p{X^{2\P}_{i}} - g\p{X^{\P}_{i}}}}\). The same convergence
behaviour of these quantities was obtained with different numbers of uniform
time-steps. Even though \(\kernel_{1}\) and \(\kernel_{2}\) are only Lipschitz
continuous and \(g\) is discontinuous, the observed weak convergence rate is
still \(\Order{\P^{-1}}\), as predicted by \cref{thm:weak-conv} when
\(\kernel_{1},\kernel_{2},\) and \(g\) were assumed to be three-times
differentiable. %
Hence, it may be that the assumptions required by \cref{thm:weak-conv} and
similar proofs in the literature can be relaxed by exploiting, for example,
the smoothness of the probability measure.

\newcommand{\negratetriangle}[5]{%
    \draw [thick]
    ({#1}, {#2})
    -- ({#3}, {#2})
    -- ({#1}, {#2 * (#3/#1)^(#4)})
    node[inner sep=0,outer sep=0, anchor=south west, scale=0.8] at (#1, #2) {#5};
}

\newcommand{\posratetriangle}[5]{%
    \draw [thick]
    ({#1}, {#2})
    -- ({#3}, {#2})
    -- ({#1}, {#2 * (#3/#1)^(#4)})
    node[inner sep=0,outer sep=0, anchor=south east, scale=0.8] at (#1, #2) {#5};
  }

\newlength\figureheight
\newlength\figurewidth
\setlength\figureheight{7cm}
\setlength\figurewidth{9cm}

\begin{figure}\centering
\begin{tikzpicture}
  \begin{axis}[width=\figurewidth, height=\figureheight, ]
  \addplot+[black,ybar interval,mark=no] plot table {
    x y
    -0.4545454545454545 4.8828125e-08
    -0.43434343434343425 1.46484375e-07
    -0.41414141414141414 4.8828125e-07
    -0.3939393939393939 1.3671875e-06
    -0.3737373737373737 2.685546875e-06
    -0.3535353535353535 4.931640625e-06
    -0.33333333333333326 1.1572265625e-05
    -0.31313131313131304 2.255859375e-05
    -0.2929292929292928 4.892578125e-05
    -0.2727272727272727 0.000101708984375
    -0.2525252525252525 0.00018544921875
    -0.23232323232323226 0.0003435546875
    -0.21212121212121204 0.000587890625
    -0.19191919191919182 0.000977587890625
    -0.1717171717171716 0.0015392578125
    -0.1515151515151515 0.002366552734375
    -0.13131313131313127 0.003431298828125
    -0.11111111111111105 0.004854833984375
    -0.09090909090909083 0.006572998046875
    -0.07070707070707061 0.008620849609375
    -0.050505050505050386 0.010914794921875
    -0.030303030303030276 0.013451220703125
    -0.010101010101010055 0.016107763671875
    0.010101010101010166 0.018889111328125
    0.030303030303030498 0.021703662109375
    0.05050505050505061 0.02453701171875
    0.07070707070707072 0.02722548828125
    0.09090909090909105 0.029852783203125
    0.11111111111111116 0.03244169921875
    0.1313131313131315 0.0346052734375
    0.1515151515151516 0.03673544921875
    0.1717171717171717 0.038679443359375
    0.19191919191919204 0.0403087890625
    0.21212121212121215 0.041871875
    0.2323232323232325 0.04289423828125
    0.2525252525252526 0.0438203125
    0.27272727272727293 0.044395068359375
    0.29292929292929304 0.044683544921875
    0.31313131313131315 0.0445701171875
    0.3333333333333335 0.04418095703125
    0.3535353535353536 0.0433005859375
    0.3737373737373739 0.041834814453125
    0.39393939393939403 0.039872705078125
    0.41414141414141437 0.037112939453125
    0.4343434343434345 0.033762109375
    0.4545454545454546 0.029783642578125
    0.4747474747474749 0.025247021484375
    0.49494949494949503 0.0204890625
    0.5151515151515154 0.015874853515625
    0.5353535353535355 0.011604052734375
    0.5555555555555556 0.007965673828125
    0.5757575757575759 0.005159375
    0.595959595959596 0.003073388671875
    0.6161616161616164 0.001719873046875
    0.6363636363636365 0.000884912109375
    0.6565656565656568 0.00043291015625
    0.6767676767676769 0.0001994140625
    0.696969696969697 8.2666015625e-05
    0.7171717171717173 3.4716796875e-05
    0.7373737373737375 1.4111328125e-05
    0.7575757575757578 5.126953125e-06
    0.7777777777777779 1.904296875e-06
    0.7979797979797982 5.37109375e-07
    0.8181818181818183 2.9296875e-07
  };
\end{axis}
\end{tikzpicture}
   \caption{A histogram of the values of \(\br{X^{\P}_{i}\p{1}}_{i=1}^{\P}\) which follows the
    system of SDEs in \cref{eq:X-sys} with \cref{eq:example-particle-system}.
    Here, the values were approximated using the Euler--Maruyama time-stepping
    scheme with \(64\) uniform time-steps and \(\P=2,048\). }
\label{fig:histogram}
\end{figure}
\begin{figure}\centering
\begin{tikzpicture}
  \begin{axis}[
    xlabel={$\P$},
    ymin=5e-5,
    xmode=log,ymode=log,
    width=\figurewidth,
    height=\figureheight,
    tick align=outside,
    grid={both},
    legend pos=south west,
    legend style={draw=none}
    ]
    \addplot [black, thick, mark=diamond, mark size=3, mark options={solid,fill opacity=0}]
    table {%
      x y
      16 0.013503125
      32 0.00643125
      64 0.002975
      128 0.0015796875
      256 0.0008419921875
      512 0.000401953125
      1024 0.000196337890625
    };
    \addlegendentry{\(\abs*{\E*{g\p{X^{2\P}_{i}\p{1}}
          -g\p{X^{\P}_{i}\p{1}}}}\)}

    \addplot [black, thick, mark=square, mark size=3, mark options={solid,fill opacity=0}]
    table{
      x y
      16 0.012932772916020874
      32 0.008022558087718151
      64 0.005226210022281428
      128 0.0035632286243281654
      256 0.0024633436948714835
      512 0.001722599750887295
      1024 0.0012048219869378164
    };
    \addlegendentry{\(\abs*{\E*{
          \p*{X^{2\P}_{i}\p{1} -X^{\P}_{i}\p{1}}^{2}}}^{1/2}\)}
    \label{rates/strong}

    \negratetriangle{128}{0.0003}{350}{1}{\(\Order{\P^{-1}}\)}
    \posratetriangle{1000}{0.0035}{350}{0.5}{\(\Order{\P^{-1/2}}\)}
  \end{axis}
\end{tikzpicture}
   \caption{Convergence rates \(X^{\P}_{i}\p{1}\) which follows the system of
    SDEs in \cref{eq:X-sys} with \cref{eq:example-particle-system} and \(
    g\p{x} = \psi_{0}\p{10 \abs{x-0.2}}\). Here, the values were approximated
    using the Euler--Maruyama time-stepping scheme with \(N=64\) time-steps.
    Note that the rates are consistent with the predicted rates in
    \cref{thm:weak-conv} even though the coefficients of the SDE do not have
    sufficient smootheness as required by the theorem.}
\label{fig:conv-rates}
\end{figure}

\FloatBarrier

\bibliographystyle{plain}

\begin{thebibliography}{10}

\bibitem{BALLY199535}
Vlad Bally and Denis Talay.
\newblock The {E}uler scheme for stochastic differential equations: error
  analysis with {M}alliavin calculus.
\newblock {\em Mathematics and Computers in Simulation}, 38(1):35 -- 41, 1995.

\bibitem{BallyTa96}
Vlad Bally and Denis Talay.
\newblock The law of the {E}uler scheme for {S}tochastic {D}ifferential
  {E}quations: {II}. {C}onvergence rate of the density.
\newblock {\em Monte Carlo Methods and Applications}, 2(2):93 -- 128, 1996.

\bibitem{bayer2020weak}
Christian Bayer, Eric~Joseph Hall, and Raúl Tempone.
\newblock Weak error rates for option pricing under linear rough volatility,
  2021.

\bibitem{bayer2010adaptive}
Christian Bayer, Anders Szepessy, and Ra\'ul Tempone.
\newblock Adaptive weak approximation of reflected and stopped diffusions.
\newblock {\em Monte Carlo Methods and Applications}, 16(1):1--67, January
  2010.

\bibitem{jordain:bias-particle}
Oumaima Bencheikh and Benjamin Jourdain.
\newblock Bias behaviour and antithetic sampling in mean-field particle
  approximations of {SDE}s nonlinear in the sense of {McKean}.
\newblock {\em ESAIM: Proceedings and Surveys}, 65:219--235, 2019.

\bibitem{cardaliaguet:mean-field-notes}
Pierre Cardaliaguet.
\newblock Notes on mean field games.
\newblock Technical report, Technical report, 2010.

\bibitem{chassagneux:weak}
Jean-Fran\c{c}ois Chassagneux, Lukasz Szpruch, and Alvin Tse.
\newblock Weak quantitative propagation of chaos via differential calculus on
  the space of measures.
\newblock {\em The Annals of Applied Probability}, 32(3), June 2022.

\bibitem{chassagneux2014probabilistic}
Jean-Fran{\c{c}}ois Chassagneux, Dan Crisan, and Fran{\c{c}}ois Delarue.
\newblock {\em A probabilistic approach to classical solutions of the master
  equation for large population equilibria}, volume 280.
\newblock American Mathematical Society, 2022.

\bibitem{crisan:approximate-spdes}
Dan Crisan and Jie Xiong.
\newblock Approximate {M}ckean--{V}lasov representations for a class of
  {SPDE}s.
\newblock {\em Stochastics An International Journal of Probability and
  Stochastics Processes}, 82(1):53--68, 2010.

\bibitem{deraynal:backward-kolmogorov}
Paul-Eric~Chaudru de~Raynal and Noufel Frikha.
\newblock From the backward {K}olmogorov {PDE} on the {W}asserstein space to
  propagation of chaos for {M}c{K}ean-{V}lasov {SDE}s.
\newblock {\em Journal de Math{\'e}matiques Pures et Appliqu{\'e}es},
  156:1--124, 2021.

\bibitem{friedman:sde-book}
Avner Friedman.
\newblock {\em Stochastic Differential Equations and Applications}.
\newblock Elsevier, 1975.

\bibitem{lukasz:lyapunov}
William R.~P. Hammersley, David {\v{S}}i{\v{s}}ka, and {\L}ukasz Szpruch.
\newblock {McKean–Vlasov SDEs under measure dependent Lyapunov conditions}.
\newblock {\em Annales de l'Institut Henri Poincaré, Probabilités et
  Statistiques}, 57(2):1032 -- 1057, 2021.

\bibitem{hoel2016construction}
H\r{a}kon Hoel, Juho H\"app\"ol\"a, and Ra\'ul Tempone.
\newblock Construction of a mean square error adaptive {Euler--Maruyama} method
  with applications in multilevel {Monte} {Carlo}.
\newblock In Ronald Cools and Dirk Nuyens, editors, {\em Springer Proceedings
  in Mathematics \& Statistics}, pages 29--86. Springer International
  Publishing, 2016.

\bibitem{hoel2014implementation}
H\r{a}kon Hoel, Erik von Schwerin, Anders Szepessy, and Ra\'ul Tempone.
\newblock Implementation and analysis of an adaptive multilevel {Monte} {Carlo}
  algorithm.
\newblock {\em Monte Carlo Methods and Applications}, 20(1):1--41, January
  2014.

\bibitem{katsoulakis:limparticle}
Markos~A. Katsoulakis and Anders Szepessy.
\newblock Stochastic hydrodynamical limits of particle systems.
\newblock {\em Communications in Mathematical Sciences}, 4(3):513--549, 2006.

\bibitem{kloden:numsde}
Peter~E. Kloeden and Eckhard Platen.
\newblock {\em Numerical Solution of Stochastic Differential Equations},
  volume~23 of {\em Applications of Mathematics (New York)}.
\newblock Springer Berlin Heidelberg, 1992.

\bibitem{kolokoltsov:book-nonlinear}
Vassili~N. Kolokoltsov.
\newblock {\em Nonlinear {Markov} Processes and Kinetic Equations}, volume 182.
\newblock Cambridge University Press, July 2010.

\bibitem{cardaliaguet:master}
Jean-Michel Lasry and Pierre-Louis Lions.
\newblock Mean field games.
\newblock {\em Japanese Journal of Mathematics}, 2(1):229--260, March 2007.

\bibitem{mischler2015new}
St\'ephane Mischler, Cl\'ement Mouhot, and Bernt Wennberg.
\newblock A new approach to quantitative propagation of chaos for drift,
  diffusion and jump processes.
\newblock {\em Probability Theory and Related Fields}, 161(1-2):1--59, December
  2013.

\bibitem{mishura:mkvlasov-existence}
Yuliya Mishura and Alexander Veretennikov.
\newblock Existence and uniqueness theorems for solutions of
  {{M}c{K}ean--{V}lasov} stochastic equations.
\newblock {\em Theory of Probability and Mathematical Statistics}, 103:59--101,
  June 2021.

\bibitem{moon2005convergence}
Kyoung-Sook Moon, Anders Szepessy, Ra\'ul Tempone, and Georgios~E. Zouraris.
\newblock Convergence rates for adaptive weak approximation of stochastic
  differential equations.
\newblock {\em Stochastic Analysis and Applications}, 23(3):511--558, May 2005.

\bibitem{mordecki2008adaptive}
E.~Mordecki, A.~Szepessy, R.~Tempone, and G.~E. Zouraris.
\newblock Adaptive weak approximation of diffusions with jumps.
\newblock {\em SIAM Journal on Numerical Analysis}, 46(4):1732--1768, January
  2008.

\bibitem{szepessy2001adaptive}
Anders Szepessy, Ra\'ul Tempone, and Georgios~E. Zouraris.
\newblock Adaptive weak approximation of stochastic differential equations.
\newblock {\em Communications on Pure and Applied Mathematics},
  54(10):1169--1214, 2001.

\bibitem{sznitman1991topics}
Alain-Sol Sznitman.
\newblock Topics in propagation of chaos.
\newblock In {\em Lecture Notes in Mathematics}, pages 165--251. Springer
  Berlin Heidelberg, 1991.

\bibitem{szpruch:antithetic}
{\L}ukasz Szpruch and Alvin Tse.
\newblock Antithetic multilevel particle system sampling method for
  {{M}c{K}ean}-{{V}lasov} {{SDE}s}.
\newblock 2019.

\bibitem{TaTub90}
Denis Talay and Luciano Tubaro.
\newblock Expansion of the global error for numerical schemes solving
  stochastic differential equations.
\newblock {\em Stochastic Analysis and Applications}, 8(4):483--509, 1990.

\bibitem{Schwerin10}
Erik von Schwerin and Anders Szepessy.
\newblock A stochastic phase-field model determined from molecular dynamics.
\newblock {\em M2AN Math. Model. Numer. Anal.}, 44(4):627--646, 2010.

\end{thebibliography}

\paragraph{Acknowledgements}
R. Tempone was partially supported by the KAUST Office of Sponsored Research
(OSR), under Award numbers URF/1/2281-01-01, URF/1/2584-01-01 in the KAUST
Competitive Research Grants Program Round 8, and the Alexander von Humboldt
Foundation, through the Alexander von Humboldt Professorship award.
 
\end{document}